\documentclass[11pt]{amsart}
\usepackage{epsfig,overpic}
\usepackage[usenames,dvipsnames]{color}
\usepackage[colorlinks=true,linkcolor=blue,citecolor=BrickRed]{hyperref}

\newtheorem{thm}{Theorem}[section]
\newtheorem{lem}[thm]{Lemma}
\newtheorem{cor}[thm]{Corollary}
\newtheorem{prop}[thm]{Proposition}

\theoremstyle{definition}
\newtheorem{definition}[thm]{Definition}

\newtheorem{note}[thm]{Note}

\theoremstyle{remark}

\newcommand{\R}{\mathbf{R}}
\newcommand{\RP}{\mathbf{RP}}
\newcommand{\Z}{\mathbf{Z}}

\newcommand{\ol}[1]{{\overline #1}}

\newcommand{\w}{\widetilde}
\newcommand{\C}{\mathcal{C}}

\renewcommand{\S}{\mathbf{S}}

\renewcommand{\tilde}{\widetilde}

\DeclareMathOperator{\inte}{int}
\DeclareMathOperator{\conv}{conv}

\DeclareMathOperator{\sign}{sign}

\begin{document}

\title[Tangent lines, inflections, and vertices]{Tangent lines, inflections, and vertices\\ of  closed  curves}

\author{Mohammad Ghomi}
\address{School of Mathematics, Georgia Institute of Technology,
Atlanta, GA 30332}
\email{ghomi@math.gatech.edu}
\urladdr{www.math.gatech.edu/$\sim$ghomi}
\date{Last Typeset \today.}
\thanks{Research of the author was supported in part by NSF Grant DMS-0336455.}

\begin{abstract}
We show that every smooth closed curve $\Gamma$  immersed in Euclidean space $\R^3$ satisfies the sharp inequality $2(\mathcal{P}+\mathcal{I})+\mathcal{V}\geq 6$ which relates the numbers $\mathcal{P}$ of pairs of parallel tangent lines, $\mathcal{I}$ of inflections (or points of vanishing  curvature),  and $\mathcal{V}$ of  vertices (or points of  vanishing torsion) of $\Gamma$. We also show that $2(\mathcal{P^+}+\mathcal{I})+\mathcal{V}\geq 4$, where $\mathcal{P}^+$ is the number of pairs of concordant parallel tangent lines. The proofs, which employ curve shortening flow with surgery, are based on  corresponding inequalities  for the numbers of double points, singularities, and inflections  of closed  curves in the real projective plane $\RP^2$ and the sphere $\S^2$ which intersect every closed geodesic. These findings extend  some classical  results  in  curve theory including  works of M\"{o}bius, Fenchel, and Segre, which is also known as Arnold's  ``tennis ball theorem".
\end{abstract}

\maketitle

\tableofcontents

\section{Introduction}
To every point of a regular closed space curve  $\gamma\colon\S^1\simeq\R/2\pi\to\R^3$ 
there corresponds a  matrix  $m:=(\gamma',\gamma'',\gamma''')$  composed of the first three derivatives of $\gamma$. 
Naturally one is interested in degeneracies of $m$, which come in two geometric flavors: \emph{inflections} or points of vanishing curvature, and \emph{vertices} or points of vanishing torsion.  Here we show how inflections and vertices are related to the global arrangement of the tangent lines of $\gamma$, which leads us to extend some  well known theorems of classical curve theory.
To state our first result, let us say that  a pair of parallel tangent lines of $\gamma$ are \emph{concordant} if the corresponding tangent vectors  point in the same direction, and are \emph{discordant} otherwise. A vertex  is called \emph{genuine} if it has a connected open neighborhood where the torsion, or equivalently $\det(m)$, vanishes just once and changes sign.

\begin{thm}\label{thm:main}
Let $\gamma\colon\S^1\to\R^3$ be a regular $\C^{3}$  closed curve, and $\mathcal{P}$ be the number of  pairs of distinct  points in $\S^1$ where tangent lines of $\gamma$  are parallel. Further let   $\mathcal{I}$ and $\mathcal{V}$ be the numbers of inflections and vertices of $\gamma$ respectively. Then 
\begin{equation}\label{eq:1}
2(\mathcal{P}+\mathcal{I})+\mathcal{V}\geq 6.
\end{equation}
Furthermore, if $\mathcal{P}^+(\leq \mathcal{P})$ denotes  the number of \emph{concordant} pairs of   parallel tangent lines, then
\begin{equation}\label{eq:1.5}
2(\mathcal{P}^++\mathcal{I})+\mathcal{V}\geq 4.
\end{equation}
Finally, if $\gamma$ has only finitely many vertices, then $\mathcal{V}$ may denote the number of \emph{genuine} vertices.
\end{thm}

Thus if a closed space curve has nonvanishing curvature and torsion (which is a generic condition with respect to the $\C^1$-topology), then it must have at least three pairs of parallel tangent lines including at least two concordant pairs. This observation is meaningful because there exist closed space curves, called \emph{skew loops}, without any pairs of parallel tangent lines.  Skew loops, which were first studied by  Segre \cite{segre:global}, may be constructed in each knot class \cite{wu:knots}, and have been the subject of several recent works \cite{ghomi:shadow, ghomi:shadowII, ghomi:tangents, ghomi&tabachnikov, ghomi&solomon, solomon:central} due in part to their interesting  connections with quadric surfaces. We should also mention that \eqref{eq:1.5} had been observed by Segre \cite{segre:global}  in the special case where $\mathcal{P^+}=\mathcal{I}=0$, while \eqref{eq:1} appears to be entirely new.

As is often the case in space curve theory, Theorem \ref{thm:main} follows from a reformulation of  it in terms of spherical curves, which is of interest on its own. To derive  this deeper result first note that, assuming $\gamma$ has  unit speed,
 the  \emph{tantrix} and \emph{curvature} of $\gamma$ are given by $T:=\gamma'$, and $\kappa:=\|T'\|$ respectively.
 Thus inflections of $\gamma$  correspond precisely to \emph{singularities} of $T\colon\S^1\to\S^2$, or points where $T'$ vanishes. Otherwise,
the \emph{geodesic curvature} of $T$ in $\S^2$, given by  
\begin{equation*}
k:=\frac{\langle T'',T\times T'\rangle}{\kappa^3}=\frac{\tau}{\kappa},
\end{equation*}
is well-defined, where $\tau$ is the torsion of $\gamma$.
 Thus vertices of $\gamma$ correspond precisely to \emph{(geodesic) inflections} of $T$, as is well known  \cite{fenchel}. Further, the tangent lines of $\gamma$ are parallel at $t$, $s\in\S^1$ provided that $T(t)=\pm T(s)$, and are concordant if $T(t)=T(s)$. It is also well known that a spherical curve forms the tantrix of a space curve, if, and only if, it contains the origin of $\R^3$ in the relative interior of its convex hull \cite{fenchel, ghomi:knots}. Via these observations,  inequalities \eqref{eq:1} and \eqref{eq:1.5} follow respectively from inequalities \eqref{eq:2} and \eqref{eq:2.5} below. In continuing the analogy with Theorem \ref{thm:main}, we  say  that an inflection is \emph{genuine} if it has a connected open neighborhood where curvature vanishes just once and changes sign.

\begin{thm}[The Main Result]\label{thm:main2}
Let $\gamma\colon\S^1\to\S^2$ be a  $\C^{2}$ closed curve, $\mathcal{S}$ be the number of singular points of $\gamma$, $\mathcal{I}$ be the number of its inflections, and $\mathcal{D}$ be the number of  pairs of points $t\neq s\in\S^1$ where $\gamma(t)=\pm \gamma(s)$. Suppose that   $\Gamma:=\gamma(\S^1)$ contains the origin in  its convex hull. Then 
\begin{equation}\label{eq:2}
2(\mathcal{D}+\mathcal{S})+\mathcal{I}\geq 6.
\end{equation}
Further, if $\mathcal{D}^+ (\leq\mathcal{D})$ denotes the number of  pairs of points $t\neq s\in\S^1$ where $\gamma(t)= \gamma(s)$, then
\begin{equation}\label{eq:2.5}
2(\mathcal{D}^++\mathcal{S})+\mathcal{I}\geq 4.
\end{equation}
Furthermore, if $\Gamma$ is \emph{symmetric}, i.e., $\Gamma=-\Gamma$, then
\begin{equation}\label{eq:2.75}
2(\mathcal{D}^++\mathcal{S})+\mathcal{I}\geq 6.
\end{equation}
Finally, if $\gamma$ has only finitely many inflections, then $\mathcal{I}$ may denote the number of \emph{genuine} inflections.
\end{thm} 

Inequality \eqref{eq:2} does not seem  to have appeared before, not even in the special case where $\mathcal{D}=\mathcal{S}=0$. Inequalities \eqref{eq:2.5} and \eqref{eq:2.75}, on the other hand, extend some well known results. Indeed,
when $\mathcal{D}^+=\mathcal{S}=0$, \eqref{eq:2.5} is a theorem of Segre \cite{segre:tangents,weiner:inflection}, a special case of which was rediscovered by Arnold and famously dubbed ``the tennis ball theorem" \cite{arnold:plane, arnold:topological,angenent:inflection}. Further, when $\mathcal{S}=0$, \eqref{eq:2.5} follows from works of Fenchel \cite{fenchel2} according to Weiner  \cite[Thm. 1]{weiner:inflection}. Furthermore, when $\mathcal{D}^+=\mathcal{S}=0$, \eqref{eq:2.75} is  equivalent to a classical theorem of M\"{o}bius \cite{mobius,angenent:inflection,thorbergsson&umehara}, who showed that simple closed noncontractible curves in the real projective plane $\RP^2:=\S^2/\{\pm1\}$ have at least  $3$ inflections. More generally, Theorem \ref{thm:main2} has the following quick implication for curves in $\RP^2$:

\begin{thm}\label{cor:main3}
Let $\gamma\colon\S^1\to\RP^2$ be a  $\C^{2}$ closed curve, $\mathcal{S}$ be the number of singular points of $\gamma$, $\mathcal{I}$ be the number of its inflections, and $\mathcal{D}$ be the number of  pairs of points $t\neq s\in\S^1$ where $\gamma(t)=\gamma(s)$. If $\gamma$ is noncontractible, then
\begin{equation}\label{eq:3}
2(\mathcal{D}+\mathcal{S})+\mathcal{I}\geq 3.
\end{equation}
On the other hand, if $\gamma$ is contractible, and its image intersects every closed geodesic,  then
\begin{equation}\label{eq:3.5}
2(\mathcal{D}+\mathcal{S})+\mathcal{I}\geq 6.
\end{equation}
\end{thm} 
\begin{proof}
If $\gamma$ is noncontractible, then there is  a closed symmetric curve $\ol\gamma\colon\S^1\to\S^2$ which double covers $\gamma(\S^1)$ when composed with the standard projection $\pi\colon\S^2\to\RP^2$. So \eqref{eq:2.75} immediately implies \eqref{eq:3}. On the other hand,  if $\gamma$ is contractible, then there is a closed curve $\ol\gamma\colon\S^1\to\S^2$ such that $\pi\circ\ol\gamma=\gamma$.  Thus $\ol\gamma$ has the same number of inflections and singularities as $\gamma$ has, since $\pi$ is a local isometry. Further, double points of $\gamma$ correspond to the self intersections of $\ol\gamma(\S^1)$, and to the intersections of $\ol\gamma(\S^1)$ with $-\ol\gamma(\S^1)$. So \eqref{eq:2} implies \eqref{eq:3.5} once we note that if $\gamma(\S^1)$ intersects every closed geodesic in $\RP^2$, then $\ol\gamma(\S^1)$ does the same in $\S^2$, and thus contains the origin in its convex hull.
\end{proof}

The rest of this work will be devoted to 
proving Theorem \ref{thm:main2}, which will be achieved through a series of reductions. First, in Section \ref{sec:hemisphere}, we use some basic convexity theory and a lemma for spherical curves in the spirit of Fenchel \cite{fenchel2} to prove Theorem \ref{thm:main2} in the special case where $\mathcal{D}^+=\mathcal{S}=0$ and $\Gamma$ lies in a hemisphere. This generalizes a result of Jackson \cite{jackson:bulletin}, also studied by Osserman \cite{osserman:vertex},   which is an extension of the classical $4$-vertex-theorems of Mukhopadhyaya \cite{mukhopadhyaya} and Kneser \cite{kneser} for planar curves (Note \ref{note:osserman}). Next, in Section \ref{sec:flow}, we employ the curve shortening flow after Angenent \cite{angenent:inflection} and Grayson \cite{grayson} to show that in the case where $\mathcal{D}^+=\mathcal{S}=0$, we may continuously transform $\Gamma$  until it just lies in a hemisphere. This transformation does not increase any of the quantities $\mathcal{D}$, $\mathcal{D}^+$, $\mathcal{I}$ and $\mathcal{S}$, and thus leads to a proof of Theorem \ref{thm:main2} in the case  $\mathcal{D}^+=\mathcal{S}=0$, via the hemispherical case already considered in Section \ref{sec:hemisphere}. In particular, we obtain new proofs of M\"{o}bius's and Segre's theorems at this stage. Then it suffices to show that one may remove the double points and singularities of $\Gamma$ without increasing the sums $2(\mathcal{D}+\mathcal{S})+\mathcal{I}$ or $2(\mathcal{D}^++\mathcal{S})+\mathcal{I}$, and while keeping the origin in the convex hull of $\Gamma$. The desingularization and surgery procedures which we need here will be developed in Sections \ref{sec:singularity} and \ref{sec:double}. Then in Section \ref{sec:proof} we will be able to synthesize  these results to complete the proof. Finally we construct a number of examples in Section \ref{sec:examples}  which will demonstrate the optimality  of all the above inequalities.

The study of inflections and vertices, which  dates back to Klein's estimate for algebraic curves \cite{klein} and M\"{o}bius's $3$-inflections-theorem, has spawned a vast literature of interesting though often isolated results. One bright focal point, however, has been Segre's theorem, which implies 
the works of  Mukhopadhyaya and Kneser (Note \ref{note:stereo}), as well as that of M\"{o}bius (Lemma \ref{lem:segremobius}). It also leads to more recent results such as Sedykh's four-vertex-theorem for space curves  (Note \ref{note:sedykh}), and has connections to contact geometry, as noted by Arnold \cite{arnold:plane,arnold:topological}.  For other recent developments  and references see  \cite{ghomi:verticesA,ghomi:verticesB,gluck:notices,ovsienko&tabachnikov,thorbergsson&umeharaII,panina} and note that vertices are also known as \emph{flattening points} \cite{arnold:flattening, uribe-vargas:flattening}. Some open problems will be discussed in Notes \ref{note:sedykh}, \ref{note:fr}, and \ref{note:normals}.

\begin{note}\label{note:o}
Inequalities \eqref{eq:2} and \eqref{eq:2.5} for spherical curves are somewhat more general than their counterparts, inequalities \eqref{eq:1} and \eqref{eq:1.5}, for space curves. Indeed, as we mentioned earlier, the tantrix of a space curve will always contain the origin in the relative \emph{interior} of its convex hull \cite{ghomi:knots}, whereas in Theorem \ref{thm:main2} the origin may lie on the boundary of the convex hull. These critical cases require special care  throughout this work.
\end{note}

\begin{note}\label{note:sedykh}
The term ``vertex"  is used in the literature to refer not only to zero torsion points of space curves, but also to local extrema of the geodesic curvature of a curve in a Riemannian surface. This terminology is motivated by the fact that for spherical curves critical points of geodesic curvature correspond to points of vanishing torsion. Indeed if $p$ is a critical point of the geodesic curvature of $\Gamma\subset\S^2$, then $\Gamma$ has contact of order $3$ with its osculating circle $C$ in $\S^2$. But $C$ coincides with the osculating circle of $\Gamma$ as a curve in $\R^3$. Thus $\Gamma$ has contact of order $3$ with its osculating plane at $p$, and so has zero torsion at that point. The double use of the term vertex may be further justified by  Sedykh's theorem \cite{sedykh:vertex} which states that if a simple closed space curve lies on the boundary of its convex hull, then it must have (at least) four vertices (zero torsion points); see also \cite{sedykh:discrete, rs:torsion}. This is analogous to Mukhopadhyaya's classical four-vertex theorem for convex planar curves. See  the paper of Thorbergsson and  Umehara \cite{thorbergsson&umehara} for a short proof of Sedykh's result based on Segre's theorem. We should also mention a related open problem due to Rosenberg \cite{rosenberg:constant}: \emph{must every closed space curve bounding a surface of positive curvature have four vertices?} This question is related to a problem of Yau \cite[Problem 26]{yau:problems2} on characterizing boundaries of positively curved surfaces, which in turn is motivated by rigidity questions for closed surfaces in $\R^3$. See \cite{agw} for more references in this area.
\end{note}

\section{Preliminaries}\label{sec:prelim}
We begin by developing our basic terminology, and recording some fundamental observations which will be useful  throughout the paper. 

\subsection{Basic terminology}
Here $\R^n$ denotes the $n$-dimensional Euclidean space with origin $o$, standard inner product $\langle \cdot,\cdot\rangle$, and  norm $\|\cdot\|:=\langle\cdot,\cdot\rangle^{1/2}$. The  (closed) unit ball and sphere in $\R^n$ will be denoted by $B^n$ and $\S^{n-1}:=\partial B^n$ respectively. We often set $B:=B^2$, or more generally let $B$ denote a metric ball (or disk) in a Riemannian surface.
Let $I$ denote a connected one-dimensional manifold, which for  convenience we  identify  either with an interval in  $\R$, or  $\S^1\simeq\R/2\pi$. A   \emph{curve} is a  locally injective mapping $\gamma\colon I\to M$, where $M$ is some manifold. If $\gamma$ is injective, then we say that it is $\emph{simple}$.  A pair of  curves $\gamma_1$, $\gamma_2$ are considered \emph{topologically equivalent} if there is a homeomorphism $\theta\colon I\to I$ such that $\gamma_2=\gamma_1\circ\theta$. A point $t_0\in I$ is called a \emph{singular point} of $\gamma$ provided that the  differential map $\gamma':=d\gamma$  is either not well-defined or is singular at $t_0$. Otherwise $t_0$ will be called a \emph{regular point} of $\gamma$. Accordingly we say that $p\in M$ is a \emph{regular value} of $\gamma$ provided that every point of $\gamma^{-1}(p)$ is regular; otherwise,  $p$ is a \emph{singular value}.

It will often be convenient to  represent a class of topologically equivalent curves $[\gamma]$ by its image. More precisely, we say that $\Gamma\subset M$ is a  curve provided that there is a fixed  equivalence class of  curves  $\gamma\colon I\to M$ associated to $\Gamma$ such that $\gamma(I)=\Gamma$. The elements of $[\gamma]$ will then be called \emph{the (admissible) parametrizations} of $\Gamma$. We say that $\Gamma$ is $\C^k$ if it admits a $\C^k$ parametrization; furthermore, if this parametrization is regular, then we say that $\Gamma$ is $\C^k$-\emph{regular} (or $\C^k$-immersed). We say
$\Gamma$ is \emph{simple} if it has a simple parametrization. If a point $p$ of a $\C^k$ curve $\Gamma$ is a singular value of all $\C^k$ parametrizations of it, then $p$ is called a ($\C^k$) \emph{singular} \emph{point} of $\Gamma$; otherwise, $p$ is  a \emph{regular point} of $\Gamma$.
If $\gamma^{-1}(p)$ is a singleton, we say that $p$ is a \emph{simple point} of $\Gamma$; otherwise $p$ will be called a \emph{multiple point}, and if $\gamma^{-1}(p)$ consists of precisely two points, then $p$ will be a \emph{double point}. Finally, we say that $\Gamma$ is an \emph{arc} when $I\subset\R$ is a compact interval, and $\Gamma$ is  \emph{closed} when  $I=\S^1$.

\subsection{The maximum principle}
 If $\Gamma$ is a $\C^2$-regular  curve in a Riemannian surface $M$, and $N$ is a unit normal vector field along $\Gamma$, then the  \emph{geodesic curvature} of $\Gamma$ at a point $p\in\Gamma$ with respect to $N$ is defined as 
\begin{equation*}
k(p):=\big\langle \nabla_{T(p)} T, N(p)\big\rangle,
\end{equation*}
where $T$ is a (continuous)  unit tangent vector field along $\Gamma$, $\nabla$ is the covariant derivative, and $\langle\cdot,\cdot\rangle$ denotes the metric on $M$.  We say that $p$ is an \emph{inflection} of $\Gamma$ if $k(p)=0$, and $\Gamma$ is a \emph{geodesic} if $k$ vanishes identically.  

By \emph{near} a point, we mean in a sufficiently small open neighborhood of that point.
Let $\Gamma\subset M$ be a simple $\C^1$-regular curve, $p$ be an interior point of $\Gamma$, and $B\subset M$ be a metric ball centered at $p$. If $\Gamma$ is an arc or a closed curve, and $B$ is sufficiently small, then $\Gamma$ meets $\partial B$ only at its end points. Thus $B-\Gamma$ will consist of precisely two components, by Jordan's curve theorem. The closure of each of these components will be called  a \emph{side} of $\Gamma$ near $p$.
The following basic fact, which is  a version of the maximum principle,  will be invoked a number times in the following pages.

\begin{lem}\label{lem:mp}
Let $M$ be a Riemannian surface with nonnegative curvature, and $\Gamma\subset M$ be a $\C^1$-regular simple arc which is $\C^2$-regular in its interior. Further let  $k$ be the geodesic curvature of $\Gamma$ in its interior with respect to a unit normal vector field $N$ along $\Gamma$, and $C$ be a geodesic arc which is tangent to $\Gamma$ at an end point $p$. Suppose that  $\Gamma$ lies on one side of $C$ near $p$, and $N(p)$ points to the side of $C$ which contains (resp. does not contain) $\Gamma$ near $p$. Then either $\Gamma$ coincides with $C$ near $p$, or else $k>0$ (resp. $<0$) at some point in every neighborhood of $p$.
\end{lem}
\begin{proof}
Let $\pi\colon\Gamma\to C$ be the nearest point projection. After replacing $\Gamma$ with a smaller subarc $pq$ we may assume that $\pi$ is well-defined (by the tubular neighborhood theorem) and one-to-one (by the inverse function theorem), since $\Gamma$ is tangent to $C$ at $p$. So we may assume that $\Gamma$ is a ``graph" over $C$. Further,  we may assume that $\Gamma$ lies in a  metric ball $B\subset M$ centered at $p$, meets $\partial B$ only at $q$, and $C$ is also a simple arc in $B$ which meets $\partial B$ only at its end points. There is then a component of $C-\{p\}$ which together with $\Gamma$ forms a simple $\C^1$-regular arc, say $\ol\Gamma$, which lies in $B$ and intersects $\partial B$ only at its end point; see Figure \ref{fig:disk}. By assumption $C$ lies on one side of $\ol\Gamma$ in $B$. Let $N$ be the unit normal vector field along $\Gamma$ which points into the side of $\ol\Gamma$ containing $C$. Then $N(p)$ points to the side of $C$ which does not contain $\Gamma$.

 \begin{figure}[h] 
   \centering
   \begin{overpic}[height=1in]{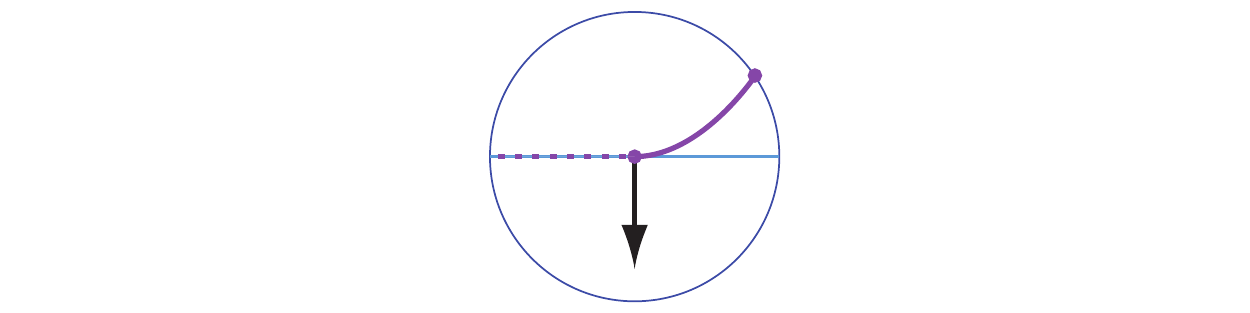} 
    \put(58,9){$C$}
   \put(55,16){$\Gamma$}
    \put(50,15){$p$}
      \put(62,19){$q$}
      \put(46,4){$N$}
   \end{overpic}
   \caption{}
   \label{fig:disk}
\end{figure}

Suppose that $\Gamma$ does not coincide with $C$ near $p$. Then  we may assume that $q\not\in C$. Let $C_\epsilon\subset B$ be the  curve at the distance $\epsilon$ away from $C$ which lies on the same side of $C$  as $\Gamma$; see Figure \ref{fig:slope}. Then $C_\epsilon$ intersects $pq$ for sufficiently small $\epsilon>0$. Let $q'$ be the first point of $\Gamma$, as we traverse it starting from $p$, which lies on $C_\epsilon$, and $p'$ (which may coincide with $p$) be the last point in $pq'$  which lies on $C$. 
Now joining the arc $p'q'$ to  the geodesic arc between $q'$ and $\pi(q')$, and the geodesic arc between $\pi(q')$ and $p'$, we obtain a  closed curve, which is simple since $\Gamma$ is a graph over $C$. Let $\Omega\subset B$ be the  region bounded by this curve. Then the exterior angles of $\partial \Omega$ at $p'$, $\pi(q')$, and $q'$, are respectively $\pi$, $\pi/2$, and $\pi/2+\delta$ for some $\delta\geq 0$.

 \begin{figure}[h] 
   \centering
   \begin{overpic}[height=1.5in]{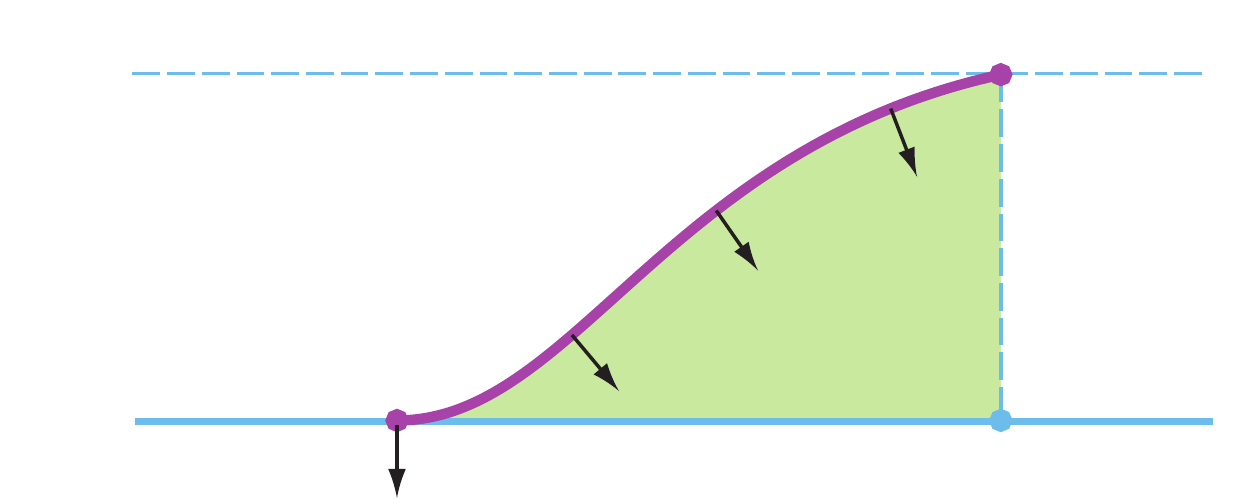} 
   \put(31,9){$p'$}
    \put(26,1){$N$}
    \put(6,6){$C$}
    \put(75,2){$\pi(q')$}
     \put(63,16){$\Omega$}
    \put(6,33){$C_\epsilon$}
   \put(45,18){$\Gamma$}
    \put(80,37){$q'$}
   \end{overpic}
   \caption{}
   \label{fig:slope}
\end{figure}

Our choice of $N$ ensures that $N$ points inside $\Omega$ along the interior of $p'q'$. Thus applying  the Gauss-Bonnet theorem to $\Omega$ yields:
\begin{equation}\label{eq:GB}
\int_{p'}^{q'}k+\int_{\Omega}K+2\pi+\delta=2\pi,
\end{equation}
where $K$ is the Gauss curvature of $M$. So, since by assumption $K\geq 0$, it follows that  $\int_{p'}^{q'} k\leq 0$. Thus either $k\equiv 0$ on the interior of $p'q'$ or else $k<0$ somewhere  between $p'$ and $q'$. Since $q'\not\in C$ and $p'q'$ is tangent to $C$ at $p'$, $k$ cannot vanish identically on $p'q'$ by the uniqueness of geodesics. So we conclude that $k<0$ somewhere on $p'q'$ which is the desired result, since $q'$ may be chosen arbitrarily close to $p$.

Finally note that if we replace $N$ with the opposite unit normal vector field along $\ol\Gamma$ (so that $N(p)$ points to the  side of $C$  containing $\Gamma$), then the geodesic curvature of the interior of $p'q'$ with respect to the inward normal of $\Omega$ will be $-k$.  Replacing $k$ with $-k$ in \eqref{eq:GB} yields that $k>0$ somewhere on $p'q'$, by repeating the argument in the last paragraph.
\end{proof}

\subsection{Convex spherical sets}
A set $X\subset\S^2$ is called \emph{convex} if every pair of points of $X$ may be joined by a  geodesic arc of length at most $\pi$ which lies in $X$. It is easy to see that $X\subset\S^2$ is convex if, and only if, the cone in $\R^3$ generated by $X$ is convex.
Thus many basic facts in the theory of convex Euclidean sets have direct analogues for convex spherical sets. In particular,  the closure of any convex spherical set is convex. We need the following local characterization for (closed) convex spherical curves, i.e., curves which bound a convex set.

\begin{lem}\label{lem:sphereconvex}
Let $\Gamma\subset\S^2$ be a simple closed curve bounding a region $\Omega$. Suppose that for every point $p\in \Gamma$, at least one of the following conditions hold:
\begin{enumerate}
\item{There exists a geodesic arc passing through $p$ which is disjoint from the interior of $\Omega$.}
\item{There exists an open neighborhood of $p$ in $\Gamma$ which is $\C^2$-regular, and has nonnegative geodesic curvature with respect to the normal vector field pointing into $\Omega$.}
\item{There is no geodesic arc in $\Omega$ which passes through $p$.} 
\end{enumerate}
Then $\Gamma$ is convex.
\end{lem}

The above lemma has some analogues for planar curves, e.g., see \cite{eggleston}. Indeed, this lemma follows from its planar version in the case where $\Omega$ lies in an open  hemisphere, via the Beltrami mapping discussed in the next section; however, we cannot a priori assume that $\Omega$ is confined to a  hemisphere, and therefore an independent proof is needed.

\begin{proof}[Proof of Lemma \ref{lem:sphereconvex}]
It suffices to show that the interior of $\Omega$, which we denote by $\inte(\Omega)$, is convex, since as we remarked above the closure of a convex spherical set is convex. 
Let $x_0$, $x_1\in \inte(\Omega)$. Since $\inte(\Omega)$ is connected and open, it is path connected. So there is a continuous map $x\colon[0,1]\to\inte(\Omega)$ such that $x(0)=x_0$, and $x(1)=x_1$. For every $t\in[0,1]$, let $G_t$ be a  geodesic arc of length at most $\pi$ which connects $x(0)$ and $x(t)$. Now let $A$ be the set of all $t\in(0,1]$ such that $G_t\subset\inte(\Omega)$. Clearly $A$ is nonempty and open. We claim that $A$ is also closed, which will complete the proof. To see this let
 $t_i\in A$ be a sequence of points converging to $t\in (0,1]$. 
 We have to show that  then $G_t\subset\inte(\Omega)$.

 First we show that $G_t\subset\Omega$.
 To see this let $B\subset\inte(\Omega)$ be a ball centered at $x(t)$, which does not contain $x(0)$; see Figure \ref{fig:guitar}. Then for $i$ sufficiently large,  $G_i:=G_{t_i}$ intersects $\partial B$ at a point $y_i$. Since $\partial B$ is compact, after replacing $y_i$ with a subsequence, we may assume that $y_i$ converge  to  $y\in\partial B$. Note that $y$ cannot be antipodal to $x(0)$,  because the lengths of $G_i$ approach the distance between $y$ and $x(0)$ plus the radius of $B$. Thus if the distance between $x(0)$ and $y$ were $\pi$, then  the lengths of $G_i$ would eventually exceed $\pi$, 
 \begin{figure}[h] 
   \centering
   \begin{overpic}[height=1.25in]{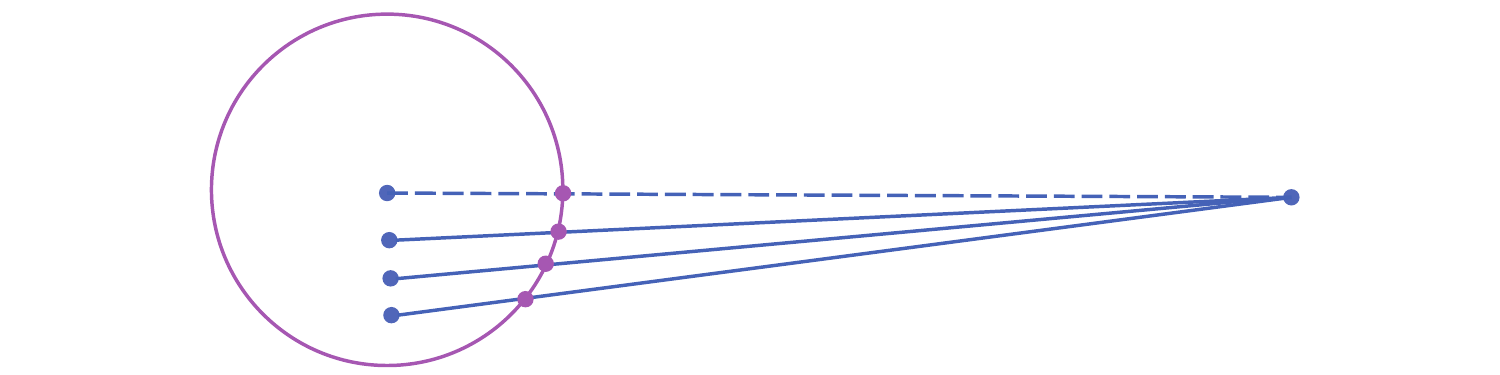} 
   \put(22,14){$x(t)$}
    \put(38,14){$y$}
     \put(87,11){$x(0)$}
      \put(19,4){$x(t_i)$}
        \put(35,3){$y_i$}
         \put(55,5){$G_i$}
           \put(55,13){$G$}
             \put(13,19){$B$}
     \end{overpic}
   \caption{}
   \label{fig:guitar}
\end{figure}
 which is not permitted. So we may assume that, for $i$ sufficiently large, there are  unique great circles $C_i$ passing through $x(0)$ and $y_i$. Consequently  $C_i$, and the arcs $G_i\subset C_i$, depend continuously on $y_i$. Let $C$ be the great circle passing through $x(0)$ and $y$. Then $C_i$ converge  to $C$ in the sense of Hausdorff distance, i.e., $C_i$ eventually lie inside any given open neighborhood of $C$. Consequently $C$ passes through $x(t)$, since $C_i$ contain $x(t_i)$ which converge to $x(t)$. Now let $G$ be the arc of $C$ which passes through $y$ and is bounded by $x(0)$ and $x(t)$. Then $G_i$ converge  to $G$. Thus, since $\Omega$ is compact,  $G\subset\Omega$. Furthermore, since  $G_i$ have length at most $\pi$, then so does $G$. Thus $G_t=G\subset\Omega$, as desired.

Now to show that  $G_t\subset\inte(\Omega)$, it suffices to check that $G_t$ is disjoint from $\partial\Omega=\Gamma$. Suppose   that $G_t$  touches $\Gamma$ at some point $p$. We claim then that $G_t\subset\Gamma$ which is the desired contradiction, since the end points of $G_t$ lie in $\inte(\Omega)$.
To establish this claim let $X:=G_t\cap\Gamma$. Then $X$ is a nonempty closed subset of $G_t$. It suffices then to show that $X$ is also open in $G_t$. To this end, let $p\in X$. Then condition (iii) in the statement of the lemma is immediately ruled out, and so either (i) or (ii) must hold. If (i) holds, then there are geodesic arcs on either side of $\Gamma$  which pass through $p$. Thus, by the uniqueness of geodesics, these arcs must coincide with each other, and consequently with $\Gamma$ near $p$, as desired. If (ii) holds, then it follows from Lemma \ref{lem:mp} (the maximum principle) that $\Gamma$ again coincides with $G_t$ near $p$, which completes the proof.
\end{proof}

Next we characterize convex spherical sets which contain a pair of antipodal points. A \emph{lune} is a subset  of $\S^2$ which is bounded by a pair of half great circles sharing the same end points.

\begin{lem}\label{lem:lune}
Let  $X$ be a proper closed convex subset of $\S^2$ which contains a pair of antipodal points. Then $X$ is a lune.
\end{lem}
\begin{proof}
Suppose $\pm p\in X$. Let $L\subset\S^2$ be the union of all half great circles $\ell_q$ which pass through $q\in X-\{\pm p\}$, and have their end points at $\pm p$. Then $X\subset L$. Further $L\subset X$, since for every  $\ell_q\in L$,  the subarcs $pq$ and $(-p)q$ of $\ell_q$ both lie in $X$ by the convexity assumption. So $L=X$, and it remains only to check that $L$ is a lune. To see this, first note that $L$ is closed since $X$ is closed. Next, let $C\subset \S^2$ be the great circle which lies on the plane which is orthogonal to the line passing through $\pm p$,  and define $\pi\colon\S^2-\{\pm p\}\to C$  by $\pi(q):=\ell_q\cap C$. Note that $\pi(L)=C\cap L$.  So $\pi(L)$ is an arc, or a closed connected proper subset of $C$, since $L$ is closed, convex, and a proper subset of $\S^2$. Consequently $L$ is a lune, since $L=\pi^{-1}(\pi(L))\cup\{\pm p\}$. Indeed $L$ is bounded by $\ell_r\cup\ell_s$ where $r$ and $s$ are the end points of $\pi(L)$.
\end{proof}

\subsection{The Beltrami map}\label{subsec:beltrami}
Let $M$ and  $\ol M$ be a pair of Riemannian surfaces,  $f\colon M\to \ol M$ be a diffeomorphism, $\Gamma\subset M$ be a $\C^2$-regular simple curve, $p\in\Gamma$, $N$ be a unit normal vector field along $\Gamma$, and $k$ be the corresponding geodesic curvature. Further let $\ol\Gamma:=f(\Gamma)$, $\ol N$ be the unit normal vector field along $\ol\Gamma$ which lies on the same side of $\ol\Gamma$ as $df(N)$, and $\ol k$ be the corresponding geodesic curvature of $\ol\Gamma$. We say that $f$ \emph{preserves the sign of geodesic curvature}, provided that $k(p)> 0$ (resp. $< 0$) if and only if $\ol k(\ol p)> 0$ (resp. $\ol k(\ol p)< 0$). By a \emph{hemisphere} in this work we always mean a closed hemisphere. A pair of curves are said to have \emph{contact of order $n$} at an intersection point provided that they admit regular parametrizations whose derivatives agree up to order $n$ at that point. 
The maximum principle (Lemma \ref{lem:mp}) quickly yields:

\begin{lem}\label{lem:beltrami}
Let $M$, $\ol M$ be Riemannian surfaces with nonnegative curvature, and $f\colon M\to\ol M$ be a diffeomorphism which preserves geodesics. Then $f$ preserves the sign of geodesic curvature. In particular, for any open hemisphere $H$ centered at a point $p\in\S^2$, the radial projection $\pi\colon H\to T_p\S^2\simeq\R^2$ from the center of $\S^2$ preserves the sign of  geodesic curvature.
\end{lem}
\begin{proof}
Note that $p\in\Gamma$ is an inflection, if, and only if, there passes a geodesic $C$ through  $p$ which has contact of order $2$ with $\Gamma$. Since $f$ is a diffeomorphism, it preserves the order of contact, and so it will preserve inflections if it preserves geodesics. Further, by Lemma \ref{lem:mp}, $k(p)>0$ (rep. $<0$) if and only if $k(p)\neq 0$, and, near $p$, $\Gamma$ lies  in the same (resp.  opposite) side of $C$ where $N(p)$ points, which are all properties  preserved by $f$.
\end{proof}

The above lemma essentially states that, as far as this work is concerned, the local geometry of spherical curves is  the same as that of planar curves, which will be a convenient observation utilized in the following pages.
We should also point out that the last sentence in Lemma \ref{lem:beltrami} is well known, e.g., see \cite{weiner:rend}, although we are not aware of an explicit or concise proof. The projection from the center of the sphere is also called \emph{gnomic projection} or \emph{Beltrami map}  \cite{docarmo&warner}.

\begin{note}\label{note:stereo}
In addition to the Beltrami map, the stereographic projection $\pi\colon\S^{2}-\{(0,0,1)\}\to\R^2$ is also useful for studying spherical curves, because it preserves their vertices. Indeed,  a point of a curve in $\R^2$ or $\S^2$ is a vertex (local extremum of geodesic curvature), if, and only if, it lies locally on one side of its osculating circle at that point. It remains to note then that $\pi$ preserves osculating circles, because it maps circles to circles and preserves the order of contact between curves.
This observation implies  that Segre's theorem, which is generalized by inequality \eqref{eq:2.5}, is itself a generalization of the classical theorem of Kneser which states that simple $\C^2$-regular closed curves $\Gamma\subset\R^2$ have (at least) $4$ vertices (earlier, Mukhopadhyaya had obtained this result for convex curves). Indeed, we may assume, after a dilation and translation, that $\Gamma$ intersects $\S^1$ in a pair of antipodal points $\pm p$.  Then the spherical curve $\ol\Gamma:=\pi^{-1}(\Gamma)$ also includes $\pm p$ and thus contains the origin in its convex hull. So, by Segre's theorem, $\ol\Gamma$ has $4$ inflections, and therefore $4$ vertices (there must be a local extremum of geodesic curvature $k$ in between every pairs of zeros of $k$). So $\Gamma$ must have $4$ vertices as well.
\end{note}

\section{Inflections of  Curves Inscribed in Great Circles}\label{sec:hemisphere}
Here we use  the observations of  the last section to prove 
Theorem \ref{thm:main2}  in the special case where $\Gamma$ is simple, regular ($\mathcal{D}^+=\mathcal{S}=0$),  and  lies in a hemisphere. To this end we need the following basic lemma.  By an \emph{orientation}  of a simple curve $\Gamma$, we mean an ordering of its points (which may be cyclical in case $\Gamma$  is closed). If $\Gamma$ is oriented, then for any pairs of points $p$, $q$ of $\Gamma$, the oriented arc in $\Gamma$ which begins at $p$ and ends at $q$ (if it exists) is denoted by $pq$. If $\Gamma$ is $\C^1$, then this ordering corresponds to a unique choice of a  unit tangent vector field along $\Gamma$, which will also be called an orientation of $\Gamma$.  We say that a pair of oriented $\C^1$-regular curves are \emph{compatibly tangent} at an intersection point $p$, if they are tangent at $p$ and their orientations coincide at that point.

\begin{lem}\label{lem:hemisphere}
Let $\Gamma\subset\S^2$ be an oriented $\C^1$-regular simple arc which is $\C^2$-regular in its interior. Suppose that $\Gamma$ lies on one side of an oriented great circle $C$, intersects $C$  only at its end points $p$ and $q$,  and is compatibly tangent to  $C$ at $p$.  Further suppose that $p$ is the initial point of $\Gamma$, and
 the  length of the arc $pq$ in $C$  is at most $\pi$. Then the geodesic curvature $k$ of the interior of 
$\Gamma$ changes sign at least once. Furthermore, if $\Gamma$ is also compatibly tangent to $C$ at $q$,  then $k$  changes sign at least twice. 
\end{lem}

\begin{figure}[h] 
   \centering
   \begin{overpic}[height=1in]{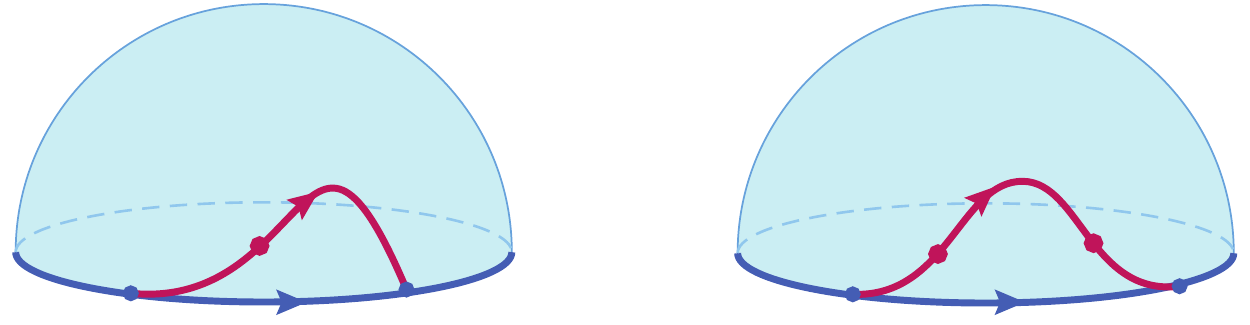} 
   \put(8,0){$p$}
        \put(32,0){$q$}
        \put(66,0){$p$}
        \put(94,0){$q$}
   \end{overpic}
   \caption{}
   \label{fig:hemispheres}
\end{figure}

\begin{proof}[Proof of Lemma \ref{lem:hemisphere}]
Let $H$ be the hemisphere determined by $C$ where $\Gamma$ lies, and $N$ be a unit normal vector field along $\Gamma$ such that $N(p)$ points inside $H$. Note that if $\Gamma$ is also compatibly tangent to $C$ at $q$, then $N(q)$ points inside $H$ as well; see Figure \ref{fig:hemispheres}. To see this, let $\ol C\subset H$ be the closed curve formed by joining $\Gamma$ to the arc $qp$ in $C$. Then $\ol C$ is  simple, so it bounds a region $\ol H\subset H$ by the Jordan curve theorem. Further, $\ol C$ is $\C^1$-regular, so we may choose a unit normal vector field $\ol N$ along $\ol C$ which always points into $\ol H$. Now note that $ N(p)=\ol N(p)$. Thus $N=\ol N$ all along $\Gamma$. In particular, $N(q)=\ol N(q)$, which shows that $N(q)$ points inside $\ol H\subset H$ as claimed. 

Next let $k$ be the geodesic curvature of the interior of $\Gamma$ with respect to $N$. Then $k> 0$ arbitrarily close to $p$ by Lemma \ref{lem:mp}. Furthermore, $k> 0$ arbitrarily close to $q$ as well, if $\Gamma$ is compatibly tangent to $C$ at $q$. So all we need is to show that $k<0$ at some point in the interior of $\Gamma$. Suppose not. Then $\ol C$ is a convex curve by Lemma \ref{lem:sphereconvex}. But $\ol C$ contains antipodal points, because the length of the arc $qp$ of $C$ is at least $\pi$ by assumption. So $\ol C$ must bound a lune by Lemma \ref{lem:lune}, which is impossible; because then   $\Gamma$ would have to be a geodesic, and therefore lie on $C$ since it is tangent to $C$ at $p$, but  $\Gamma$ may touch $C$ only at its end points by assumption.
\end{proof}

The last lemma quickly yields the following observation, which corresponds to a theorem of Osserman for planar curves (see Note \ref{note:osserman}). We say that a continuous function $f\colon\S^1\to\R$ \emph{changes sign} $n$ times, provided that there are cyclically arranged consecutive points $p_i\in\S^1$, $i\in \Z_n$,  such that $f(p_i)\neq 0$, and $\sign(f(p_i))=-\sign(f(p_{i+1}))$.

 \begin{thm}\label{thm:osserman}
Let $\Gamma\subset\S^2$ be a   $\C^2$-regular simple closed curve which is circumscribed in a great circle $C$, and has no geodesic subarcs. Suppose that $\Gamma\cap C$  includes $n$ points  not all contained in the interior of a semicircle of $C$. Then the geodesic curvature of $\Gamma$ changes sign at least $2n$ times. In particular, if $\Gamma$ has only finitely many inflections, then at least $2n$ of them must be genuine.
\end{thm}
\begin{proof}
Let $H$ be the hemisphere bounded by $C$ where $\Gamma$ lies, and $N$ be the unit normal vector field along $C$ which points into $H$. Orient $C$ by choosing a unit tangent vector field $T$ along it so that $(T,N)$ lies in a fixed orientation class of $\S^2$ (e.g.,  assume that $\det(T,N,n)>0$ where $n$ is the outward unit normal of $\S^2$).
Now let $p_i$, $i\in \mathbf{Z}_{n}$, be a cyclical ordering of the given points in $\Gamma\cap C$ with respect to the orientation of $C$. We need to show that the curvature changes sign (at least) twice on each (oriented) subarc $\Gamma_i$ with initial point $p_i$ and final point $p_{i+1}$. 

To this end first note that, since $\Gamma$ has no geodesic subarcs by assumption, there is a point $q_i\in\Gamma_i$ which does not lie on $C$. Let  $\ol\Gamma_i$ 
be the largest subarc of $\Gamma_i$ which contains $q_i$, and whose interior is disjoint from $C$. Then $\ol\Gamma_i$ will meet $C$ only at its end points $\ol{p_i}$ and $\ol{p_{i+1}}$. Further,  since $\Gamma$ is a simple closed curve,  $\ol{p_i}$ and $\ol{p_{i+1}}$ must lie in the arc $p_ip_{i+1}\subset C$ which by assumption has length at most $\pi$. Thus $\ol{p_i}\ol{p_{i+1}}\subset C$ will also have length at most $\pi$.

Now we claim that   $\ol\Gamma_i$ is  compatibly tangent to $C$ at its end points, which will complete the proof by Lemma \ref{lem:hemisphere}. This follows from the assumption that 
$\Gamma$ is simple. Indeed, let $\tilde H\subset H$ be the region bounded by $\Gamma$, and   $\tilde N$ be the unit normal vector field of $\Gamma$ which  points into $\tilde H$. Next let $\tilde T$ be the unit tangent vector field of $\Gamma$ such that $(\tilde T,\tilde N)$ lies in the preferred orientation of class of $\S^2$, which we fixed earlier.  Then $N(p)=\tilde N(p)$ at every  point $p\in\Gamma\cap C$, which in turn yields that $T(p)=\tilde T(p)$ as desired.
\end{proof}

Finally we are ready to prove the main result of this section, which follows from the last theorem via some basic convexity theory. For any set $X\subset\R^n$, let $\conv(X)$ denote the \emph{convex hull} of $X$, i.e., the intersection of all convex sets which contain $X$. 

\begin{cor}\label{cor:hemisphere}
Theorem \ref{thm:main2} holds if $\Gamma$ is simple, regular ($\mathcal{D}^+=\mathcal{S}=0$), and  lies in a hemisphere.
\end{cor}
\begin{proof}
First note that  \eqref{eq:2.75} holds trivially here, since if $\Gamma$ lies in a hemisphere and is symmetric, then it must be a great circle (and thus have infinitely many inflections). So we just need  to consider \eqref{eq:2} and \eqref{eq:2.5}.

  Since  $o\in\conv(\Gamma)$,  $o$ lies in a convex polytope $P$ with vertices on $\Gamma$, by the theorem of Carath\'{e}odory \cite{schneider:book}. Next note that,  since $\Gamma$ lies in a hemisphere it lies on one side of a plane $\Pi\subset\R^3$ which passes through $o$. In particular, $o$ lies on the boundary of $\conv(\Gamma)$, and therefore on the boundary of $P$. More specifically,  $o$ lies in $F:=\Pi\cap P$, which is either  an edge or a face of $P$. 

If  $F$ is an edge, then  the vertices  of $F$ are antipodal points of $\S^2$, since $o\in F$. Thus if  $C:=\Pi\cap\S^2$ denotes the great circle determined by $\Pi$, then $\Gamma\cap C$ consists of $2$ points which are not all contained in an open semicircle of $C$. So, by Theorem \ref{thm:osserman}, $\mathcal{I}\geq 4$, which  yields \eqref{eq:2.5}. Further, we already have  $\mathcal{D}\geq 1$ due to the existence of a pair of antipodal points on $\Gamma$. So $2\mathcal{D}+\mathcal{I}\geq 6$ which yields \eqref{eq:2}.

If, on the other hand, $F$ is a face, then it must have at least $3$ vertices. Since $o\in F$, not all these vertices can lie in an open semicircle of $C$. Thus  Theorem \ref{thm:osserman} yields that $\mathcal{I}\geq 6$, which completes the proofs of \eqref{eq:2} and \eqref{eq:2.5}.
\end{proof}

\begin{note}\label{note:osserman}
Theorem \ref{thm:osserman} extends a result of Jackson \cite[Thm. 7.1]{jackson:bulletin}, see also \cite{osserman:vertex, jackson:geodesic}, which shows that if a simple $\C^2$-regular closed curve $\Gamma\subset\R^2$ touches its circumscribing circle $C$ in $n$ points, then it must have $2n$ vertices (where  ``vertex" here means a local extremum of geodesic curvature,  and $C$ is the smallest circle which contains $\Gamma$). This  follows by applying the stereographic projection $\pi\colon\S^2-\{(0,0,1)\}\to\R^2$ along the lines discussed in Note \ref{note:stereo}.
Indeed, after a rigid motion and dilation, we may assume that $C=\S^1$. Then $\ol\Gamma:=\pi^{-1}(\Gamma)$ lies in the Southern hemisphere, and touches the equator $C$ at $n\geq 2$ points not all of which are contained in an open semicircle. So, by Theorem \ref{thm:osserman}, $\ol\Gamma$ has (at least) $2n$ inflections, and therefore $2n$ vertices. Hence $\Gamma$ has $2n$ vertices as well, since $\pi$ preserves vertices.
\end{note}

\section{Curve Shortening Flow: A New Proof of  Segre's Theorem}\label{sec:flow}
Here we apply the observations of  the last section,  as summed up in Corollary \ref{cor:hemisphere}, to prove inequality \eqref{eq:2} in the special case where $\Gamma$ is regular and simple ($\mathcal{D}^+=\mathcal{S}=0$). Without much additional effort, we will also obtain in this section proofs of  the inequalities \eqref{eq:2.5}  and \eqref{eq:2.75} in the simple regular case. These constitute new proofs of the theorems of Segre and M\"{o}bius, which were mentioned in the introduction. The main strategy here is to use the curve shortening flow \cite{chou&zhu,grayson,hass&scott,angenent:inflection} to  deform $\Gamma$ through simple $\C^2$-regular  curves, without increasing the number of inflections or antipodal pairs of points of $\Gamma$, until $\Gamma$ just lies in a hemisphere, in which case results of the previous section may be applied. The curve shortening flow is ideally suited to this task due to a number of well known remarkable properties which we summarize below for the reader's convenience. A $\C^{2,1}$ curve is one which has a $\C^{2,1}$ parametrization (i.e., a $\C^2$ parametrization with Lipschitz continuous second derivative). A simple closed curve $\Gamma\subset\S^2$ is said to \emph{bisect} $\S^2$ if the components of $\S^2-\Gamma$ have equal area.

 \begin{lem}\label{lem:flow}
 Let $\Gamma\subset\S^2$ be a simple $\C^{2,1}$-regular  closed curve. There exists a continuous one parameter family $\Gamma_t\subset\S^2$, $0\leq t<t_1$, of simple $\C^2$-regular  closed curves with $\Gamma_0=\Gamma$  such that
 \begin{enumerate}
 \item{$\Gamma_t$ converges to a point, as $t\to t_1$, unless $\Gamma$ bisects $\S^2$.} 
 \item{The number  of  inflections  of $\Gamma_t$ is nonincreasing; furthermore, if $\Gamma_0$ has only finitely many inflections, the number of \emph{genuine} inflections of $\Gamma_t$ is nonincreasing as well.}
  \item{The number  of intersections of $\Gamma_t$ with $-\Gamma_t$  is nonincreasing.}
 \end{enumerate}
 \end{lem}
 \begin{proof}
 Let $\gamma_0\colon\S^1\to\Gamma$ be a $\C^{2,1}$ regular parametrization, and
 $\gamma_t$ be the evolution of $\gamma_0$ according to the curve shortening flow, i.e., the solution to  the equation
 $$
 \frac{\partial}{\partial t}\gamma_t(s)=k_t(s)N_t(s),
 $$
 where $N_t$ is a continuous choice of unit normal vector fields along $\gamma_t$, and $k_t$ is the corresponding geodesic curvature.
Grayson's theorem \cite{grayson} states that $\Gamma_t:=\gamma_t(\S^1)$ is a simple $\C^\infty$-regular closed curve, for some interval $0<t<t_1$, which converges either to a closed geodesic or a point as $t\to t_1$. Let $\Omega_t$ be a continuous choice of regions in $\S^2$ bounded by $\Gamma_t$, $A(t)$ be the area of $\Omega_t$, and suppose that $N_t$ always points into $\Omega_t$. After possibly replacing $\Omega_t$ with its complements, we may assume that $A(0)\leq 2\pi$. In particular, if $\Gamma_t$ does not bisect $\S^2$, we have $A(0)<2\pi$. 

It follows from the Gauss-Bonnet theorem that $A'(t)=A(t)-2\pi$ \cite{angenent:inflection}. Thus $A(t)$ is decreasing, and so it cannot approach $2\pi$. Consequently $\Gamma_t$ cannot converge to a great circle, and must therefore converge to a point by Grayson's theorem. This settles item (i). For item (ii)  see \cite{angenent:inflection} or \cite[Cor. 2.5]{grayson} which establish that the number of inflections never increases; furthermore, see \cite[Lem. 1.5]{grayson}, where it is shown that for $t>0$ the only times when $k_t$ and $k'_t$ both vanish at the same point are when two or more inflections merge. In particular, non-genuine  inflections exist only at isolated times before disappearing, and the non-genuine inflections of $\Gamma_0$ immediately disappear as well.
Finally, item (iii) follows from a theorem of  Angenent \cite[p. 179]{angenent:parabolic} who showed that the number of intersections of a pair of curves under the curve shortening flow is nonincreasing.
 \end{proof}

To apply the curve shortening flow we need the initial curve to be $\C^{2,1}$ according to Lemma \ref{lem:flow}; however, in Theorem \ref{thm:main2} we are assuming only that the given curve is $\C^2$. To resolve this discrepancy we need the following basic approximation result:

\begin{lem}\label{lem:smoothingk}
Let $f\colon\S^1\to\R$ be a continuous function with a finite number of zeros. Then for any given $\epsilon>0$ there exists a $\C^\infty$ function $\tilde f\colon\S^1\to\R$ such that $\|\tilde f-f\|\leq\epsilon$, and $\tilde f$ has no more zeros than $f$ has.
\end{lem}
\begin{proof}
Let $t_i\in\S^1\simeq\R/2\pi$, $i=1,\dots, n$, be the zeros of $f$, and choose $\delta>0$ sufficiently small so that  $f\neq 0$ on $[t_i-\delta, t_i)\cup (t_i, t_i+\delta]$.  Next let $\ol f\colon \S^1\to\R$ be the function which is obtained from $f$ by replacing the graph of $f$ over each of the intervals $[t_i-\delta, t_i+\delta]$ with the line segment joining the points $(t_i-\delta, f(t_i-\delta))$ and $(t_i+\delta, f(t_i+\delta))$. Choosing $\delta$ sufficiently small, we may assume that $\|f\|\leq\epsilon/4$ on $[t_i-\delta,  t_i+\delta]$, which yields that  $\| \ol f \| \leq\epsilon/4$ on $[t_i-\delta,  t_i+\delta]$ as well. So it follows that $\|\ol f-f\|\leq\epsilon/2$. Also note that $\ol f$ has no more zeros than $f$ has, $\ol f$ is linear on $[t_i-\delta, t_i+\delta]$, and $\ol f\equiv f\neq 0$ outside these neighborhoods. 

For $\lambda>0$, let $\theta_\lambda\colon\R\to\R$ be a $\C^\infty$ nonnegative function such that  $\int_\R\theta_\lambda=1$, and $\theta_\lambda\equiv 0$ outside $(-\lambda,\lambda)$. Also let  $\tilde f_\lambda:=\ol f \ast \theta_\lambda$ be the convolution of $\ol f$ with $\theta_\lambda$. Choosing $\lambda$ sufficiently small, we may assume that $\|\tilde f_\lambda-\ol f\|\leq\epsilon/2$. Thus, 
$$\|\tilde f_\lambda-f\|\leq\|\tilde f_\lambda-\ol f\|+\|\ol f-f\|=\epsilon.$$
 Further,  $\tilde f_\lambda\equiv\ol f$ on each interval $I_i:=[t_i-\delta+\lambda, t_i+\delta-\lambda]$, since $\ol  f$ is linear on $[t_i-\delta, t_i+\delta]$. So $\tilde f_\lambda$ has at most one zero in each $I_i$. Also note that $\ol f\neq 0$ outside  the interiors of $I_i$ (which is a compact region). Thus, since $\tilde f_\lambda\to \ol f$, as $\lambda\to 0$, we may choose $\lambda$ so small that $\tilde f_\lambda\neq 0$ outside $I_i$. Then $\tilde f_\lambda$ is the desired function.
\end{proof}

The last lemma leads to the next observation. Here $\|\cdot\|_{\C^2}$ denotes the standard (supremum) norm in the space of $\C^2$ maps $f\colon I\to\R^n$.

\begin{prop}\label{prop:smoothingGamma}
Let $\gamma\colon\S^1\to \S^2$ be a regular $\C^2$ closed curve with finitely many inflections. Then for every $\epsilon>0$ there is a regular $\C^\infty$  closed curve $\tilde \gamma\colon\S^1\to \S^2$ which has no more inflections than $\gamma$ has, and $\|\tilde\gamma-\gamma\|_{\C^2}\leq\epsilon$. 
\end{prop}
\begin{proof}
Suppose that $\gamma$ has constant speed, let $k$ be the geodesic curvature of $\gamma$, and $\ol k\colon\S^1\to\R$ be the smoothing of $k$ given by Lemma \ref{lem:smoothingk}. We may identify $\ol k$ with a $2\pi$-periodic function on $\R$, and assume that $\ol k(0)\neq 0$. Now let $\ol\gamma\colon\R\to \S^2$ be the curve with geodesic curvature $\ol k$ such that $\ol\gamma(0)=\gamma(0)$ and $\ol\gamma'(0)=\gamma'(0)$. Then $\ol\gamma$ is $\C^\infty$ and, choosing $\ol k$ sufficiently close to $k$, we may assume that  $\|\ol\gamma-\gamma\|_{\C^2}\leq \epsilon$ on $[-2\pi, 4\pi]$. This follows from basic ODE theory which guarantees the existence and uniqueness of $\ol\gamma$ together with  its continuous dependence on $\ol k$. Indeed, $\ol\gamma\to\gamma$ with respect to the $\C^2$ norm as $\ol k\to k$.

Next we are going to ``close up" $\ol\gamma$ by a perturbation as follows. Assume that $\delta$ is so small that $\gamma$ has no inflections on $[-\delta, \delta]$. Let $\phi\colon [0,2\pi+\delta]\to\R$ be any $\C^\infty$ nonincreasing function such that $\phi\equiv 1$ on $[0, 2\pi-\delta]$ and $\phi\equiv 0$ on $[2\pi,2\pi+\delta]$. Now define $\tilde\gamma\colon [0,2\pi+\delta]\to\R^3$ by
$$
\tilde\gamma(t):=\phi(t)\ol\gamma(t)+\big(1-\phi(t)\big)\ol\gamma(t-2\pi).
$$
Then $\tilde\gamma(t)=\ol\gamma(t)=\tilde\gamma(t+2\pi)$ on $[0,\delta]$. So setting $\tilde \gamma(t+2\pi):=\tilde \gamma(t)$, and then replacing $\tilde \gamma$ with $\tilde \gamma/\|\tilde \gamma\|$, we obtain a
  regular $\C^\infty$ closed curve $\tilde \gamma\colon\S^1\simeq\R/2\pi\to\S^2$.
Since $\ol\gamma$ may be chosen  arbitrarily $\C^2$-close to $\gamma$, it follows that $\tilde\gamma$  may be arbitrarily $\C^2$-close to $\gamma$. In particular, since $\gamma$ has no inflections on $[-\delta, \delta]$, then we may assume that $\tilde\gamma$ has no inflections on that interval either. Thus $\tilde\gamma$ has no more inflections than $\gamma$ has.
\end{proof}

Finally we need the following basic  result which will allow us to  perturb curves locally without increasing the number of their isolated inflections. This lemma, which will be invoked frequently in the following pages, applies to all curves which may be represented locally as the graph of a function $f\colon[a,b]\to\R$ (e.g., all curves which are $\C^1$ or convex). Note that 
the graph of $f$  has an inflection  at $(x, f(x))$, if, and only if, $f''(x)=0$. We say that a pair of functions $f_1$, $f_2\colon [a,b]\to\R$ have the \emph{same regularity properties} provided that $f_1$ is $\C^k$ on a subset $U\subset[a,b]$, if, and only if, $f_2$ is $\C^k$ on $U$. 

\begin{lem}\label{lem:perturb}
Let $f\colon[-a,a]\to\R$ be  a continuous function. Suppose that $f$ is $\C^2$ on $[-a,0)\cup(0,a]$, and $f''\neq 0$ except possibly at $0$. Then for any $\epsilon>0$, there is a function $\tilde f\colon [-a,a]\to\R$ with the same regularity properties as $f$, such that $\tilde f\geq f$, $\tilde f(0)>f(0)$,  $\tilde f\equiv f$ near $\pm a$,  $\|\tilde f-f\|\leq\epsilon$, and in case $f$ is $\C^2$, $\|\tilde f-f\|_{\C^2}\leq\epsilon$. Finally, $f''_\epsilon$ has no more zeros than  $f''$ has.
\end{lem}
\begin{proof}
 Let $\phi\colon [-a,a]\to\R$ be a $\C^\infty$  nonnegative function such that $\phi\equiv 0$ outside $(-a/2,a/2)$ and $\phi\equiv1$ on $[-a/4,a/4]$. Set
$
\tilde f_\delta:= f+\delta\phi.
$
It is easy to see that, for sufficiently small $\delta>0$, $\tilde f_\delta$ has all the required properties except possibly the last one. To see that the last requirement is met as well,  note that $\tilde f''_\delta\equiv f''$ outside of $A:=[-a/2,-a/4]\cup[a/4,a/2]$. Thus it suffices to check that $\tilde f''_\delta\neq 0$ on $A$. This is indeed the case since as $\delta\to 0$, $\tilde f_\delta\to f$ with respect to the $\C^2$ norm on $A$, and $f''\neq 0$ on $A$,  by assumption. 
\end{proof}

We are now ready to prove the main result of this section:

\begin{prop}\label{prop:DS}
Theorem \ref{thm:main2} holds when $\Gamma$ is simple and regular ($\mathcal{D}^+=\mathcal{S}=0$).
\end{prop}

One step in the proof of the above result worth being highlighted is the following observation which shows that M\"{o}bius's theorem follows from Segre's result:

\begin{lem}\label{lem:segremobius}
Inequality \eqref{eq:2.5} implies \eqref{eq:2.75} when $\mathcal{D}^+=\mathcal{S}=0$.
\end{lem} 
\begin{proof}
Suppose that $\Gamma=-\Gamma$. We may assume that $\Gamma$ has only finitely many inflections. Then the number of genuine inflections $\mathcal{I}\geq 4$ by \eqref{eq:2.5}. We need to show that  $\mathcal{I}\geq 6$. Since 
$\mathcal{I}$ must be even,  it suffices then to show that $\mathcal{I}>4$.  Suppose, towards a contradiction, that $\mathcal{I}=4$. Then the sign of the four arcs of $\Gamma$ in between these inflections alternate as we go around $\Gamma$. In particular, we end up with a pair of opposite arcs with the same sign. This is  the desired contradiction because antipodal reflection in $\S^2$ switches the sign of geodesic curvature, i.e., $k(-p)=-k(p)$ for every point $p\in\Gamma$.
\end{proof}

Next we complete the rest of the proof:

\begin{proof}[Proof of Proposition \ref{prop:DS}]
By Lemma \ref{lem:segremobius}, it remains to establish the inequalities \eqref{eq:2} and \eqref{eq:2.5}. First note that 
if $\mathcal{D}^+=\mathcal{S}=0$ in Theorem \ref{thm:main2}, then $\Gamma$ is  a $\C^{2}$-regular simple  closed curve.  Now if $\Gamma$ lies in a hemisphere, then Theorem \ref{thm:osserman} completes the proof. So we may assume that the origin $o$ lies in the interior of $\conv(\Gamma)$. Then, by Proposition \ref{prop:smoothingGamma}, we may assume after a perturbation that $\Gamma$ is $\C^{2,1}$, and thus curve shortening results of Lemma \ref{lem:flow} may be applied to $\Gamma$. Also note  that $\Gamma$ remains simple after this perturbation, since it is arbitrarily small with respect to the $\C^1$-topology.
Now set
 $\Gamma_0:=\Gamma$ and let $\Gamma_t$ be the corresponding flow as in Lemma \ref{lem:flow}. Since, by Lemma \ref{lem:flow}, the number of genuine inflections  and antipodal pairs of points of $\Gamma_t$ do not increase with time (not to mention that  $\Gamma_t$ remains simple), it suffices to show 
 that $\Gamma_t$ will satisfy \eqref{eq:2} and \eqref{eq:2.5} at some future time. 
 
 To take advantage of the last observation  note that, by Lemma \ref{lem:perturb}, we may assume that $\Gamma_0$ does not bisect $\S^2$ after a perturbation. Indeed this perturbation may be confined to an arbitrarily small neighborhood of a point $p$ of $\Gamma_0$ which is neither an inflection nor an element of an antipodal pair of points of $\Gamma$.  More specifically, a small neighborhood $U$ of such a point may be identified with a ball in $\R^2$ via Lemma \ref{lem:beltrami}.  Then an arc of $\Gamma_0$ containing $p$ may be identified with the graph of a function $f\colon [-a,a]\to\R$. Then Lemma \ref{lem:perturb} yields the desired perturbation of the graph of $f$, and consequently of $\Gamma_0$.

 So we may assume that
$\Gamma_t$  converges to a point as $t\to t_1$, by Lemma \ref{lem:flow}.
In particular, $\Gamma_t$ eventually lies in a  hemisphere.   Set 
\begin{equation}\label{eq:t1}
t_0:=\inf\big\{t\in[0,t_1)\mid \Gamma_t \text{ lies in a hemisphere}\big\}.
\end{equation}
Then $\Gamma_{t_0}$ lies in a hemisphere (because according to our definition hemispheres are closed). Furthermore,  $o\in\conv(\Gamma_{t_0})$. If not, then there exists a plane passing through $o$ which is disjoint from $\conv(\Gamma_{t_0})$. Consequently, $\Gamma_{t_0}$ lies in the interior of a hemisphere $H$. This yields that $\Gamma_{t_0-\epsilon}$ must have lain in the interior of $H$ as well, for some $\epsilon>0$, which contradicts \eqref{eq:t1}. Thus $\Gamma_{t_0}$ satisfies \eqref{eq:2} and \eqref{eq:2.5} by Corollary \ref{cor:hemisphere}.
 \end{proof}

\begin{note}
The curve shortening flow has been used by Angenent \cite{angenent:inflection} to prove Segre's theorem in the special case where $\Gamma$ bisects $\S^2$ (the case known as Arnold's ``tennis ball theorem"). In that setting $\Gamma$ flows to a great circle, and the Sturm theory is used to show that perturbations of great circles have at least $4$ inflections, if they contain the origin in their convex hulls. In comparison, our proof here is more elementary (as well as more general), since we avoid  using Sturm theory by an initial perturbation of the curve, which ensures that it will flow into a hemisphere, where Theorem \ref{thm:osserman} may be applied.
\end{note}

\section{Removal of  Singularities}\label{sec:singularity}
Having proved Theorem \ref{thm:main2} in the simple regular case ($\mathcal{D}^+=\mathcal{S}=0$), we now turn to the general case.
To this end we need a desingularization result  developed in this section, which  shows that  the singularities  of a $\C^2$ curve $\Gamma\subset
\S^2$ may be eliminated without increasing the sum
\begin{equation}\label{eq:sig+}
\Sigma^+(\Gamma):=2(\mathcal{D}^++\mathcal{S})+\mathcal{I}.
\end{equation}
Note that as far as proving Theorem \ref{thm:main2} is concerned, we may assume that all multiple points of $\Gamma$ are double points, for if $\gamma^{-1}(p)$ has more than $2$ elements for some $p\in\Gamma$, then $\mathcal{D}^+\geq 3$, which yields that $\Sigma^+(\Gamma)\geq 6$ and so there would be nothing left to prove. A multiple point $p$ is called a \emph{multiple singularity} if at least two of the  points in the preimage $\gamma^{-1}(p)$ are singular points of $\gamma$. If $\Gamma$ has a multiple singularity, then $(\mathcal{D}^++\mathcal{S})\geq 3$ which again yields $\Sigma^+(\Gamma)\geq 6$. Thus, as far as proving theorem \ref{thm:main2} is concerned, we may assume that $\Gamma$ has no multiple singularities. 

\begin{thm}\label{thm:singularity}
Let $M$ be a Riemannian surface of constant nonnegative curvature, $\Gamma\subset M$ be a $\C^{2}$ arc with   finitely many singularities and inflections, and $p$ be a singular point  of  $\Gamma$ which is either a simple point, or a double point, but is not a multiple singularity. Then for every open neighborhood $U$ of $p$ in $M$, which does not contain any other singularities, there is a  $\C^{2}$  arc $\tilde\Gamma$ such that $\tilde\Gamma=\Gamma$ outside $U$, $\tilde\Gamma$ is  $\C^2$-regular inside $U$, and $\Sigma^+(\tilde\Gamma)\leq\Sigma^+(\Gamma)$.
\end{thm}

The rest of this section will be devoted to proving the above theorem, which constitutes the most technical part of the paper. This  will unfold in two parts: first we consider the case of simple singularities, and then that of singularities which are double points (but are not multiple singularities). Note that since the theorem is essentially local, we may assume that  $M=\R^2$ via Lemma \ref{lem:beltrami}. Indeed, assuming that the closure of 
$U$ is a sufficiently small metric ball centered at $p$, we may identify it (isometrically) with a ball in $\R^2$, in case the curvature of $M$ is zero; otherwise, we may isometrically map the closure of $U$ to a ball in an open hemisphere of $\S^2$ after a uniform rescaling of the metric, and then map it to $\R^2$ via the Beltrami projection described in Section \ref{subsec:beltrami}. 

\subsection{Simple singularities}\label{subsec:simple}
Here we suppose that $p$ in Theorem \ref{thm:singularity} is a simple point. Then, after replacing $\Gamma$ with a subarc, we may  assume that $\Gamma$ is a simple arc without any inflections and singularities except possibly at $p$. Further we may assume that $\Gamma$ lies in a ball $B$, $p$ coincides with the center $o$ of $B$, and $\Gamma$ intersects $\partial B$ only at its end points. Then $\Gamma$ will be called a \emph{double spiral}. For easy reference, let us record that:

\begin{definition}
A simple arc $\Gamma\subset\R^2$ is a \emph{spiral}, if it is $\C^2$-regular in the complement of  a  boundary point $o$, called the \emph{vortex} of $\Gamma$, has no inflections except possibly at $o$, lies in a ball $B$ centered at  $o$, and touches $\partial B$ only at its other end point.  Further, $\Gamma$ is a \emph{double spiral} if it is  the union of two spirals, called the \emph{arms} of $\Gamma$,  which lie in the same ball $B$, and intersect only at their common vortex $o$ in the center of $B$.
\end{definition}

Note that we do not assume that  a spiral $\Gamma$ has any regularity at its vortex $o$. In particular, the tangent lines of $\Gamma$ may not converge to  a limit at $o$ ($\Gamma$ may rotate endlessly about $o$). Further, $\Gamma$ may have infinite length.
Recall that the closure of each component  of $B-\Gamma$ is called a side $S$ of $\Gamma$. We say  $S$ 
 is   \emph{proper} provided that  there exists no line segment in $S$ passing through $o$.   If   $\Gamma$ has no proper sides, then it must coincide with a line segment in a neighborhood of $o$, which implies that the arms of $\Gamma$ must contain inflections---a contradiction. So \emph{every double spiral has at least one proper side}. 
 
 Double spirals may be categorized into three types as follows. By a \emph{principal normal} vector field of a double spiral $\Gamma$ we mean a unit normal vector field along $\Gamma-\{o\}$ such that the geodesic curvature of $\Gamma$ with respect to $N$ is always positive. Since by assumption $\Gamma-\{o\}$ has no inflections, $N$ always exists.
 If  $N$ only points into one side $S$ of $\Gamma$, then   $\Gamma$ is  called either \emph{convex} or \emph{concave} according to whether or not $S$ is proper. Otherwise (if $N$  changes sides),  we say that $\Gamma$ is  $\emph{semiconvex}$.  Examples of these three types  of double spirals are depicted in Figure \ref{fig:wedges}.
 \begin{figure}[h] 
   \centering
   \includegraphics[height=1.1in]{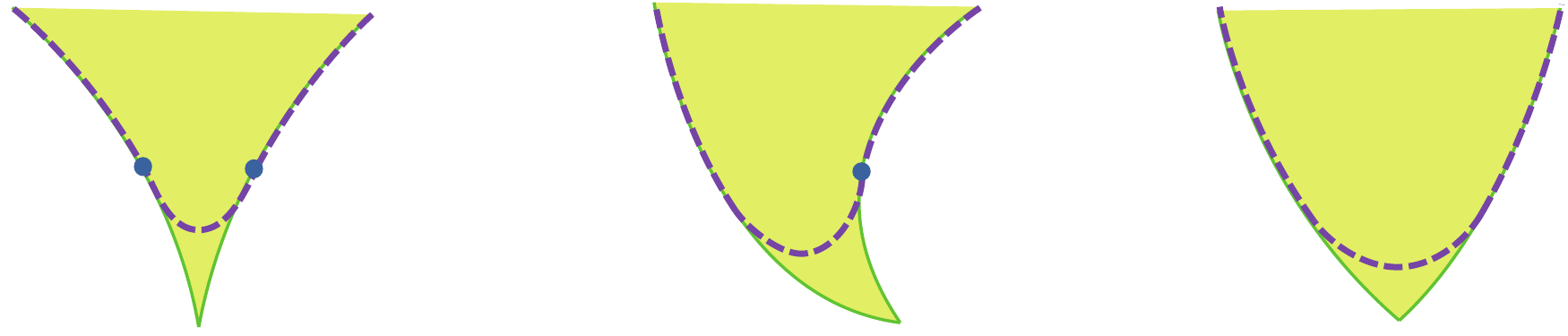} 
   \caption{}
   \label{fig:wedges}
\end{figure}
As this figure suggests, in each case we may remove the singularity which may exist at the vortex $o$, by a perturbation of $\Gamma$ near $o$, to obtain a simple $\C^2$-regular  curve with at most two inflections. This is the content of the next result:
 
 \begin{prop}\label{prop:singularity}
 Let $\Gamma$ be a double spiral with vortex $o$, and $S$ be a proper side of $\Gamma$. Then there exists a simple $\C^2$-regular arc $\w\Gamma\subset S-\{o\}$  which coincides with $\Gamma$ near its end points, and has at most two inflections. Furthermore, if $\Gamma$ is semiconvex or  convex, then we may require that $\tilde\Gamma$ have precisely $1$ or  $0$ inflections respectively.
 \end{prop}
 
The case of simple singularities in Theorem \ref{thm:singularity} is an immediate consequence of the above proposition. The proof of this  proposition  will be divided into three parts:
 
 \subsubsection{The convex case}\label{subsec:convex}
First we prove Proposition \ref{prop:singularity} in the case where $\Gamma$ is  a convex double spiral. 
In general a simple curve $\Gamma$ in $\R^2$ or $\S^2$ is called \emph{convex} if it lies on the boundary of a convex body (convex set with interior points), and the following lemma confirms that this is indeed the case with convex double spirals.

 \begin{lem}\label{lem:Sconvex}
A  convex double spiral is a convex curve.
 \end{lem}
 \begin{proof}
 If $\Gamma$ is a convex double spiral, then it has a proper side $S$ into which the principal normals of 
$\Gamma-\{o\}$ point. Since  $\Gamma\subset\partial S$, we just need to check that $S$ is convex. To this end it suffices in turn to note that each point of $\partial S$ which lies in the interior of an arm of $\Gamma$ satisfies the hypothesis (ii) of Lemma \ref{lem:sphereconvex}, and all other points of $\partial S$ satisfy hypothesis (iii) of that Lemma (which has a direct analogue in $\R^2$ via Lemma \ref{lem:beltrami}). In particular $o$ satisfies hypothesis (iii) since $S$ is proper.
 \end{proof}
 
So a convex  double spiral $\Gamma$   may be represented as the graph of a convex function near its vortex $o$. Now, by Lemma \ref{lem:perturb}, we may perturb $\Gamma$ near $o$ to obtain a convex double spiral $\ol\Gamma$ which lies in the convex side $S$ of $\Gamma$, coincides with $\Gamma$ near its end points, and its vortex $\ol o$ lies in the interior of $S$. Then there is a ball $\ol B$ centered at $\ol o$ which is disjoint from $\Gamma$. To complete the proof of the convex case of Proposition \ref{prop:singularity}, it suffices then to show that we may perturb $\ol\Gamma$ within  $\ol B$ to obtain a $\C^2$-regular curve without inflections, which is precisely the content of the next lemma:

\begin{lem}\label{lem:convexsmoothing}
Let $\Gamma$ be a convex double spiral, and $U\subset\R^2$ be an open neighborhood of $\Gamma$. Then there exists a $\C^2$-regular simple arc $\tilde\Gamma\subset U$ which coincides with $\Gamma$ near its end points and has no inflections.
\end{lem}
\begin{proof}
After a rigid motion, we may identify a neighborhood of the vortex $o$ of $\Gamma$ with the graph of a convex function $f\colon[-a,a]\to \R$. The desired curve $\tilde\Gamma$ is then obtained by replacing the graph of $f$ with that of the function $\tilde f\colon [-a,a]\to \R$ constructed as follows.

As in the proof of Lemma \ref{lem:smoothingk}, let $\theta_\lambda\colon\R\to\R$ be a $\C^\infty$ nonnegative function such that  $\int_\R\theta_\lambda=1$, and $\theta_\lambda\equiv 0$ outside $(-\lambda,\lambda)$. Next let $f_\lambda\colon [-a+\lambda, a-\lambda]\to\R$ be the convolution of $ f$ with $\theta_\lambda$, i.e., set
$$
f_\lambda(x):=\int_{-a}^a f(x-y)\theta_\lambda(y)dy.
$$
Then, $f''_\lambda(x)=\int_{-a}^a f''(x-y)\theta_\lambda(y)dy$. So, since $f''>0$ almost everywhere, it follows that $f''_\lambda>0$. Now let $\phi\colon[-a,a]\to\R$ be a $\C^\infty$  nonnegative function such that $\phi\equiv 1$ on $[-a/4,a/4]$,  and $\phi\equiv 0$ on $[-a,-a/2]\cup [a/2,a]$. Assuming $\lambda<a/2$, and extending $f_\lambda$ to $[-a,a]$ arbitrarily, we may define $\tilde f_\lambda\colon [-a,a]\to\R$ by
$$
\tilde f_\lambda (x):=\big(1-\phi(x)\big)  f(x)+\phi(x)f_\lambda(x).
$$
We claim that, for sufficiently small $\lambda$, $\tilde f_\lambda$ is the desired function. 

To check this claim note that
as $\lambda\to 0$, $f_\lambda\to f$, and thus $\tilde f_\lambda\to  f$. So, choosing $\lambda$ sufficiently small, we may assume that $\|\tilde f_\lambda-f\|$ is small enough so that $\tilde\Gamma_\lambda\subset U$, where $\tilde\Gamma_\lambda$ is the curve which is obtained from $\Gamma$ by replacing the graph of $f$ with that of $\tilde f_\lambda$. Further note that $\tilde f_\lambda''>0$ on $[-a,-a/2]\cup [-a/4,a/4]\cup [a/2,a]$, since $f''$, $f_\lambda''>0$ on this region. Next we show that for $\lambda$ sufficiently small, $\tilde f_\lambda''>0$ on $A:=[-a/2,-a/4]\cup [a/4,a/2]$ as well, and thus $\tilde\Gamma_\lambda$ is free of inflections, which would complete the proof. To see this note that on $A$
$$
\tilde f_\lambda''= f''+\phi'' (f_\lambda- f)+2\phi' (f_\lambda- f)'+\phi\, (f_\lambda- f)''.
$$
Further, $f_\lambda\to f$ with respect to the $\C^2$ norm on $A$. Thus, since $\phi$ is independent of $\lambda$, $\tilde f_\lambda''\to  f''$ on $A$. So $\tilde f_\lambda''>0$ on $A$ for small $\lambda$, because $ f''>0$ on $A$.
\end{proof}

\begin{note}
Lemma \ref{lem:convexsmoothing} shows that an isolated  inflection which lies in an open neighborhood where the curvature does not change sign may be removed by a local  $\C^2$-small perturbation.
\end{note}
  
 \subsubsection{The $\C^1$-regular case}
Next we show that Proposition \ref{prop:singularity} holds when $\Gamma$ admits a regular $\C^1$  parametrization, or equivalently, it has a well defined tangent line which varies continuously. To this end first note that:

\begin{lem}\label{lem:convexorsemi}
If $\Gamma$ is a $\C^1$-regular double spiral, then $\Gamma$ is either convex or semiconvex.
\end{lem}
\begin{proof}
We may suppose that the principal normal vector field $N$ of $\Gamma-\{o\}$ always points to one side of $\Gamma$ (otherwise $\Gamma$ is semiconvex and there is nothing to prove). Then $N$ may be continuously extended to $o$ since $\Gamma$ is $\C^1$. Let $S$ be the side of $\Gamma$ into which $N$ points. We need to show that $S$ is proper, i.e., there exists no line segment $L$ in $S$ which passes through $o$. Suppose otherwise. Then  $L$ lies on one side of $\Gamma$, and $N(o)$ points into the same side. So, by Lemma \ref{lem:mp}, the geodesic curvature of $\Gamma-\{o\}$ with respect to $N$ must be negative, which is not possible since $N$ is the principal normal.
\end{proof}

 Now, since the convex case has already been treated,  we may  assume that $\Gamma$ is semiconvex. In this case we once again use Lemma \ref{lem:perturb} to perturb $\Gamma$ inside the given proper side $S$, so that the vortex enters the interior of $S$, just as we had done in Section \ref{subsec:convex}. Then all that remains is to show:

\begin{lem}\label{lem:semiconvex}
Let $\Gamma$ be a $\C^1$-regular semiconvex double spiral, and $U\subset\R^2$ be an open neighborhood of $\Gamma$. Then there exists a $\C^2$-regular simple arc $\tilde\Gamma\subset U$ which coincides with $\Gamma$ near its end points, and has precisely one inflection.
\end{lem} 
\begin{proof}
Much like the beginning of the proof of Lemma \ref{lem:convexsmoothing}, we may identify a neighborhood of the vortex $o$ of $\Gamma$ with the graph of a function $f\colon[a,b]\to \R$. Replacing the graph of $f$ with that of an appropriate function $\tilde f\colon [a,b]\to \R$ then yields the desired curve $\tilde\Gamma$. We may assume, after a rigid motion, that $f(0)=f'(0)=0$, $f''>0$ on $(0,b]$, and $f''<0$ on $[a,0)$. Then we construct $\tilde f$ as follows.

 Consider the cubic function $g(x):=\lambda (x^3+x)$, where $\lambda>0$. Choosing $\lambda$ sufficiently small, we can make sure that $g(b)<f(b)$, and $g(a)>f(a)$. On the other hand, since $g(0)=f(0)$, while $g'(0)=\lambda>0=f'(0)$, it follows that $g>f$ on some interval $(0,b')$. Let $0<b'<b$ be the largest number so that $g>f$ on $(0,b')$. Similarly, there exists $a<a'<0$ such that $(a',0)$ is the largest open interval ending at $0$ where $g<f$. See Figure \ref{fig:graphs}. 
 
 \begin{figure}[h] 
   \centering
   \begin{overpic}[height=1.5in]{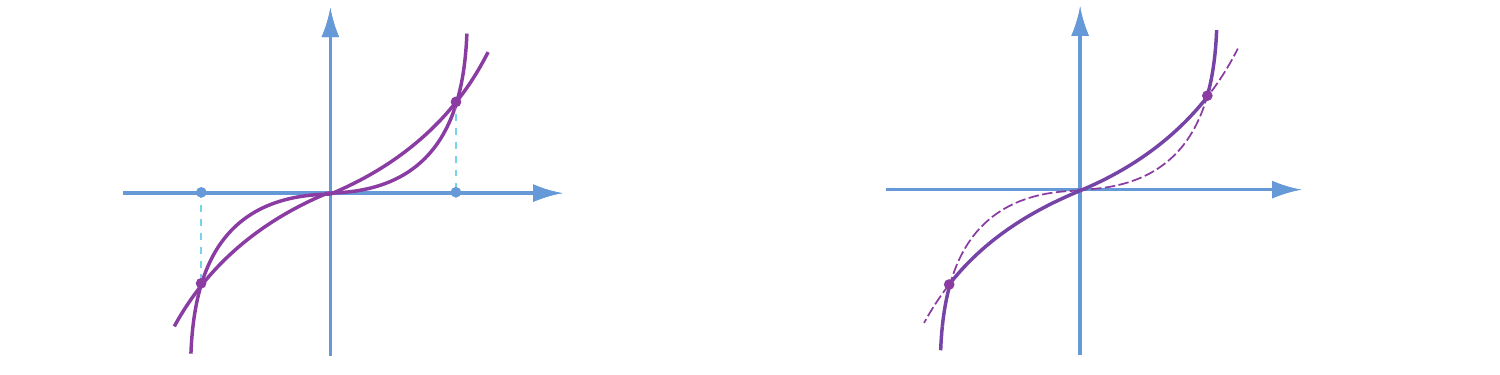} 
 \put(30,10){$b'$}
  \put(11,10){$a'$}
    \put(33,20){$g$}
     \put(29,22){$f$}
      \put(78,22){$\ol f$}
    \end{overpic}
   \caption{}
   \label{fig:graphs}
\end{figure}

 Now define $\ol f\colon [a,b]\to\R$ by setting $\ol f:=g$ on $[a',b']$ and $\ol f:=f$ otherwise. Then $\ol f$ is  convex on $[0,b]$ and   concave on $[a,0]$, because the maximum (resp. minimum) of two convex (resp. concave) functions is convex (resp. concave). Note that $\ol f$ is $\C^2$ everywhere except near $a'$ and $b'$, and the only place where $\ol f''$ vanishes is at $0$. Now we may smooth $\ol f$ near  $a'$ and $b'$ by invoking Lemma \ref{lem:convexsmoothing} in small neighborhoods of these points to  obtain the desired function $\tilde f$. Finally note that, after replacing $\Gamma$ with a subarc, we may assume that $U$ is the interior of an open ball centered at $o$, in which case it is clear that the graph of $\tilde f$ will lie within $U$ as well.
\end{proof}

 \subsubsection{Other cases} It remains now to prove Proposition \ref{prop:singularity} in the case where $\Gamma$ is neither convex nor $\C^1$-regular. To this end first note that

 \begin{lem}\label{lem:spiraltangent}
Let $\Gamma$ be a spiral with vortex $o$, $p$ be an interior point of $\Gamma$, and $\ell$ be the tangent line of $\Gamma$ at $p$. Then the subarc $po$ of $\Gamma$ intersects $\ell$ only at $p$.
 \end{lem}
 \begin{proof}
 Suppose, towards a contradiction, that there exists a point $q$ on $po$, other than $p$ itself which lies on $\ell$. Since $\Gamma$ has no inflections at $p$ the arc $po$ is disjoint from $\ell$ near $p$ (by Taylor's theorem). Thus we may assume that $q$ is the ``first" point of $po$ after $p$ which intersects $\ell$, assuming that $\Gamma$ is oriented so that $o$ is the ``end point".
  After a rigid motion we may assume that $\ell$ coincides with the $x$ axis,  the  tangent vector of $\Gamma$ at $p$ (which is consistent with the orientation of $\Gamma$) is parallel to the positive direction of the $x$-axis, and $\Gamma$ lies above the $x$-axis near $p$, see Figure \ref{fig:whale}.
  \begin{figure}[h] 
   \centering
   \begin{overpic}[height=0.9in]{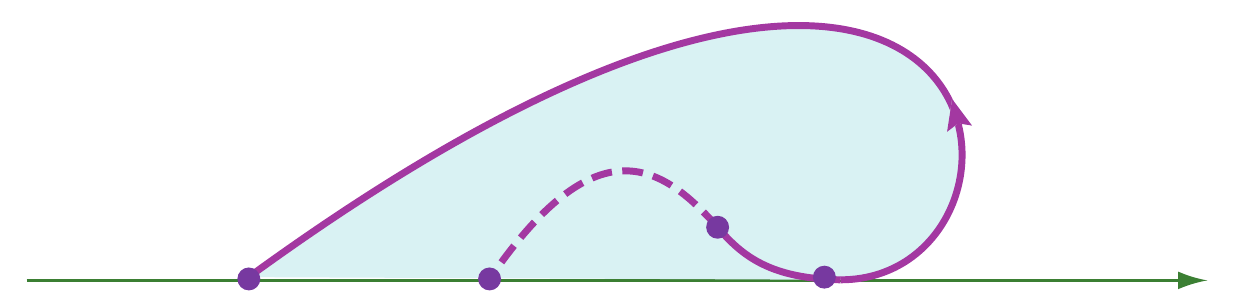}
   \put(65,-2){$p$}
    \put(18,-2){$q$}
    \put(57,8){$r$}
    \put(38,-2){$s$}
   \end{overpic}
   \caption{}\label{fig:whale}
\end{figure}
 Now it follows from Lemma \ref{lem:hemisphere} that $q$ lies to the ``left" of $p$ (with respect to the positive directions of the $x$-axis). Indeed, if $q$ were to lie on the ``right" of $p$, then the inverse of the Beltrami projection applied to the subarc $pq$ of $\Gamma$ would  yield a spherical curve with at least one inflection in its interior according to Lemma \ref{lem:hemisphere}---a contradiction.
 So connecting the end points of the subarc $pq$ with a straight line yields a simple closed curve $C$ which is $\C^1$-regular in the complement of $q$.  Now note  that  if a point $r$ of $\Gamma$ precedes $p$ and is sufficiently close to $p$, then $r$ lies in the interior of the region bounded by $C$, because by assumption $\Gamma$  lies above the $x$-axis near $p$. So the arc $o'r$ of $\Gamma$, where $o'$ is the initial point of $\Gamma$, must intersect $C$ at some point, because $o'$ lies on $\partial B$, the boundary of the ball associated to the spiral $\Gamma$, which lies outside $C$. Since $\Gamma$ is simple, $o'r$ can intersect $C$ only in the interior of the line segment $pq$. Let $s$ be the last point of $\Gamma$ prior to  $r$ which lies on $pq$. Then $s$ lies to the left of $p$, and again it follows from Lemma \ref{lem:hemisphere} that $sp$ contains an inflection, which is the contradiction we were seeking.
  \end{proof}

 The last lemma  yields:
 
 \begin{lem}\label{lem:spiralball}
 Let $\Gamma$ be a spiral with vortex $o$. Then either the tangent lines of $\Gamma$ converge to a line at $o$, or else $\Gamma$ intersects every ray emanating from $o$ infinitely often. In particular, if the interior of $\Gamma$ is disjoint from some ray emanating from $o$, then $\Gamma$ is $\C^1$-regular.
 \end{lem}
 \begin{proof}
 Let $L$ denote the (possibly infinite) length of $\Gamma$, 
 $\gamma\colon[0,L)\to\Gamma$ be a unit speed parametrization of $\Gamma-\{o\}$ with $\lim_{t\to L}\gamma(t)=o$, and
 $T:=\gamma'\colon[0,L)\to\S^1$ be the corresponding tantrix. Then we may write $T(t)=(\cos(\theta(t)), \sin(\theta(t)))$, for some continuous function $\theta\colon [0,L)\to\R$. Since $\Gamma-\{o\}$ has no inflections, $\theta'\neq 0$.  For definiteness, we may suppose that $\theta'>0$, or $\theta$ is strictly increasing. Then either $\theta$ is bounded above, or it increases indefinitely. If $\theta$ is bounded above, then it converges to a limit, and it follows that $T$ may be continuously extended to $[0,L]$. This would show that $\Gamma$ is $\C^1$-regular. On the other hand, if $\theta$ is not bounded above, then $T$  winds around  $\S^1$ indefinitely. By Lemma \ref{lem:spiraltangent},
$\gamma(t)/\|\gamma(t)\|$ and $T(t)$ can never be parallel. So $\gamma$  has to wind around $o$ indefinitely as well.
 \end{proof}

Now we can complete the proof of Proposition \ref{prop:singularity}. Recall  that here we are assuming that the  double spiral $\Gamma$   is not convex. So on at least one of the arms of $\Gamma$, say $\Gamma_1$, the principal normals point outward with respect to the given proper side $S$. Take a point $p$ of $\Gamma_1$ other than $o$. Then, for some $r>0$, there exists a circle $C$ of radius $r$ which lies in $S$ and passes through $p$.  Let $q\in op$ be the point closest to $o$ along $\Gamma_1$ such that there passes a circle $C$ of radius $r$ through $q$ which lies in $S$; see Figure \ref{fig:wedgeball}. 

 \begin{figure}[h] 
   \centering
   \begin{overpic}[height=1.5in]{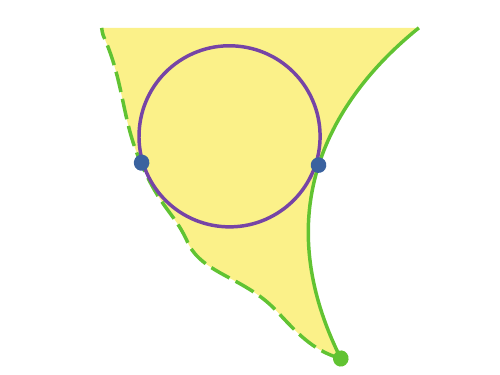}
   \put(66,37){$q$}
    \put(20,37){$q'$}
     \put(43,32){$\sigma$}
     \put(70,0){$o$}
      \put(82,60){$\Gamma_1$}
      \end{overpic}
   \caption{}\label{fig:wedgeball}
\end{figure}

If $q=o$, then  each arm of $\Gamma$ is $\C^1$-regular, by Lemma \ref{lem:spiralball}. So $S$ has a well-defined angle $\theta$ at $o$. Since by assumption $S$ is a proper side of $\Gamma$, $\theta\leq\pi$. On the other hand, if $q=o$, then  $\theta\geq\pi$. So we conclude that $\theta=\pi$, which yields that $\Gamma$ is $\C^1$-regular. Then Lemma \ref{lem:semiconvex} completes the proof.

So we may suppose that $q\neq o$. Then $C$ must intersect the other arm $\Gamma_2$ of $\Gamma$ at some point (other than $o$); because,  by Lemma \ref{lem:spiraltangent}, $C$ does not intersect the arc $oq$ of $\Gamma_1$ at any point other than $q$. Thus if $C$ were disjoint from $\Gamma_2$, then we could slide $C$ closer to $o$ along $\Gamma_1$ which would be a contradiction. 

Let $q'$ be the closest point of $\Gamma_2$ to $o$ which intersects $C$. Then $q$ and $q'$ determine an arc $\sigma$ in \ $C$ such that the union of $\sigma$ with the arc $q'q \subset\Gamma$ forms a simple closed curve which  bounds a region outside $C$. Replacing the arc $q'q$ of $\Gamma$ with $\sigma$ then yields a $\C^1$-regular arc $\ol\Gamma$. Note that $\ol\Gamma$ has at most only two points near which it may not be $\C^2$-regular, namely $q$ and $q'$. Since $\ol\Gamma$ is $\C^1$-regular, however, we may smooth away these singularities by Lemma \ref{lem:semiconvex}, and possibly Lemma \ref{lem:convexsmoothing}, which completes the proof. Indeed, note that $\ol\Gamma$ is semiconvex near $q$, so removing that singularity will cost precisely one inflection. On the other hand, by Lemma \ref{lem:convexorsemi}, $\ol\Gamma$ is either convex or semiconvex near $q'$, so removing $q'$ will cost either $0$ or $1$ inflections respectively. The former case occurs only when $\Gamma$ is semiconvex, and the latter occurs only when $\Gamma$ is concave. So the case where $\Gamma$ is semiconvex will cost  $1$ inflection, and the concave case will cost $2$ inflections, as claimed.

\subsection{Singularities which are double points}\label{subsec:double}
To complete the proof of Theorem \ref{thm:singularity}, it remains only to consider the case where $p$ is a singular point of $\Gamma$ which is also a double point (but is not a multiple singularity). To see why this would suffice, note that as in Section \ref{subsec:simple}, assuming $U$ is sufficiently small, we may identify it with the interior of a ball $B\subset\R^2$ centered at  $o$. Then $B$ will contain precisely two branches or subarcs of $\Gamma$, say $\Gamma_1$ and $\Gamma_2$. Furthermore, $\Gamma_1$ is $\C^2$-regular, while $\Gamma_2$ has precisely one singularity at $o$, which is the only point where the two arcs meet. We show that we may perturb $\Gamma_1$ and $\Gamma_2$ so as to ``separate" the singularity and the double point. This perturbation will leave $\Gamma_1$ and $\Gamma_2$ fixed near their end points, and does not increase the numbers of intersections, singularities, or inflections of these curves. Then we may remove the singularity, which is now simple, via Proposition \ref{prop:singularity}, to complete the proof of Theorem \ref{thm:singularity}. In summary, all we need is to show:

\begin{prop}\label{prop:seperate}
Let $\Gamma_i$, $i=1$, $2$, be  simple  arcs  which lie in a ball $B\subset\R^2$ centered at $o$,  and intersect each other only at $o$. Suppose that $\Gamma_1$ is $\C^2$-regular,  $\Gamma_2$ is $\C^2$-regular in the complement of $o$, and $\Gamma_i$ have only finitely many inflections. Then there are simple arcs $\tilde \Gamma_i$  in $B$ which coincide with $\Gamma_i$  near their end points, have  the same regularity properties as $\Gamma_i$,  have no  more inflections than $\Gamma_i$ have, and intersect   at most at one point which is not a singularity. 
\end{prop}

To establish this proposition, we may assume that $\Gamma_1$ and $\Gamma_2$ cross each other at $o$, i.e., one of the arms of $\Gamma_2$, say $\ol\Gamma_2$, lies on one side of $\Gamma_1$ while the other arm  lies on the opposite side of $\Gamma_1$. Indeed if $\Gamma_2$ lies entirely on one side of $\Gamma_1$, then we may perturb $\Gamma_1$ into the opposite side via Lemma \ref{lem:perturb} which will quickly remove the double point, and leave behind a simple singularity. 

Second, since by assumption $\Gamma_i$  are allowed to have only finitely many inflections, we may assume, after replacing $\Gamma_i$  by smaller subarcs (and reducing the size of $B$ accordingly), that these curves have no inflections except possibly at $o$. 

Third,  we may  assume after a rigid motion that $\Gamma_1$ is tangent to the $x$-axis at $o$, and $\ol\Gamma_2$ lies ``above" $\Gamma_1$ near $o$; see  Figure \ref{fig:pies}(a).
Also note that, by Lemma \ref{lem:spiralball},  $\ol\Gamma_2$ is $\C^1$-regular. Further, after a perturbation we may assume that $\ol\Gamma_2$ is not tangent to the $x$-axis, by the next lemma; see  Figure \ref{fig:pies}(b). More specifically, near $o$ we may represent $\ol\Gamma_2$  as the graph of a function $f$ satisfying the hypothesis of the following lemma. Then replacing the graph of $f$ with the graph of the function $\tilde f$ given by this lemma yields the desired perturbation of  $\ol\Gamma_2$.

\begin{figure}[h] 
   \centering
  \begin{overpic}[height=0.9in]{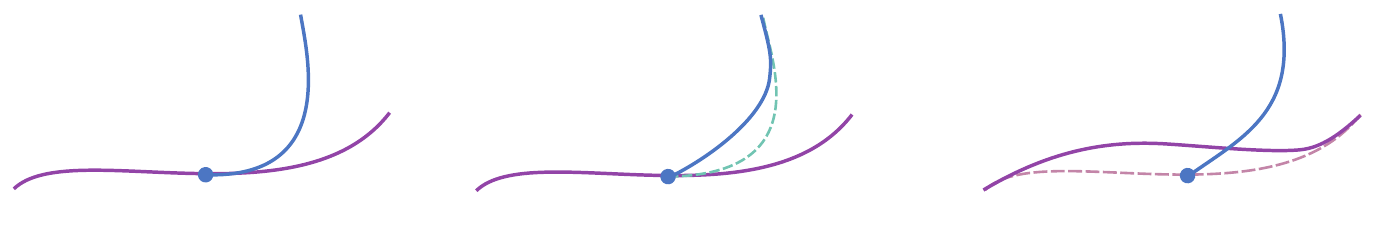}
   \put(12,-1){$(a)$}
   \put(48,-1){$(b)$}
   \put(85,-1){$(c)$}
   \end{overpic}

   \caption{}\label{fig:pies}
\end{figure}

\begin{lem}\label{lem:rotatespiral}
Let $f\colon [0,a]\to\R$ be a $\C^1$ function which is $\C^2$ on $(0,a]$. Suppose that $f(0)=f'(0)=0$, and $f$, $f''>0$ on $(0,a]$. Then there exists a $\C^1$ function $\tilde f\colon [0,a]\to\R$, which is $\C^2$ on $(0,a]$, $\tilde f(0)=0$, $\tilde f'(0)>0$, $\tilde f\geq f$,  $\tilde f''>0$ on $(0,a]$, and $\tilde f\equiv f$ near $a$.
\end{lem}
\begin{proof}
Let $\phi\colon [0,a]\to\R$ be a $\C^\infty$ nonnegative function such that $\phi(x)\equiv x$ on $[0,a/3]$, and $\phi\equiv 0$ on $[2a/3,a]$. Set $\tilde f:=f+\epsilon\phi$ for constant $\epsilon>0$. Then $\tilde f\geq f$. Further, $\tilde f'(0)>0$, $\tilde f''\equiv f''$ on $(0,a/3]\cup [2a/3,a]$, and $\tilde f''\to f''$ with respect to the $\C^2$ norm on $[a/3, 2a/3]$. So, for small $\epsilon$, we  have $\tilde f''> 0$ on $(0,a]$.  
 \end{proof}

So  we may assume that $\ol\Gamma_2$ meets $\Gamma_1$ transversely at $o$. Then by Lemma \ref{lem:perturb} we may perturb $\Gamma_1$, by a sufficiently small amount with respect to the $\C^1$-norm, so that it intersects $\ol\Gamma_2$ at only one point which is different from $o$, and does so transversely; see Figure \ref{fig:pies}(c). This will settle Proposition \ref{prop:seperate}, which will in turn complete the proof of Theorem \ref{thm:singularity}.

\section{Surgery on Regular Double Points}\label{sec:double}
In this section we prove an analogue of the main result of the last section, Theorem \ref{thm:singularity}, for multiple points. Recall that, as far as proving Theorem \ref{thm:main2} is concerned, we may assume that all multiple points of $\Gamma$ are double points.

\begin{thm}\label{thm:doublepoints}
Let $M$ be a Riemannian surface of constant nonnegative curvature, $\Gamma\subset M$ be a  $\C^{2}$-regular closed curve  with finitely many  inflections and singularities, and $p$ be an isolated double point of $\Gamma$ which is not a multiple singularity. Then for every open neighborhood $U$ of $p$ which does not contain any other multiple points of $\Gamma$, there is a $\C^2$-regular closed curve $\tilde\Gamma\subset M$ such that $\tilde\Gamma=\Gamma$ outside $U$, $\tilde\Gamma$ has no multiple points inside $U$,
 and $\Sigma^+(\tilde\Gamma)\leq\Sigma^+(\Gamma)$.
\end{thm}

By Theorem \ref{thm:singularity} we may assume that $p$ is not a singular point of $\Gamma$. Further, as in Section \ref{subsec:double}, we may identify $U$ with the interior of a ball $B\subset\R^2$ centered at $o$, and let $\Gamma_i$, $i=1$, $2$, be the two branches  of $\Gamma$ in $B$. Also, again as in Section \ref{subsec:double}, we may assume that $\Gamma_1$ and $\Gamma_2$ cross each other due to Lemma \ref{lem:perturb}. Finally note that here $\Gamma_i$ are $\C^2$-regular, and by the next result we may assume that their point of intersection is not an inflection of either curve:

\begin{prop}\label{prop:separate}
Let $\Gamma_i$, $i=1$, $2$, be  simple  $\C^2$-regular arcs  which lie in a ball $B\subset\R^2$ centered at $o$,  and intersect each other only at $o$. Suppose that $\Gamma_i$ have no inflections except possibly at $o$. Then there are simple $\C^2$-regular arcs $\tilde \Gamma_i$  in $B$ which coincide with $\Gamma_i$  near their end points,  have no  more inflections than $\Gamma_i$ have,  are transversal to each other, and  intersect  at most at one point which is not an inflection of either curve. 
\end{prop}

Assume for now that the above proposition holds. Then after replacing $\Gamma_i$ with smaller subarcs, and replacing $B$ by a smaller ball centered at the new intersection point of $\Gamma_i$, we may assume that $\Gamma_i$ have no inflections, and intersect $\partial B$ only at their end points. Next note that  $B-(\Gamma_1\cup\Gamma_2)$ consists of precisely $4$ components, see Figure \ref{fig:x}. 
\begin{figure}[h] 
   \centering
   \includegraphics[height=1.2in]{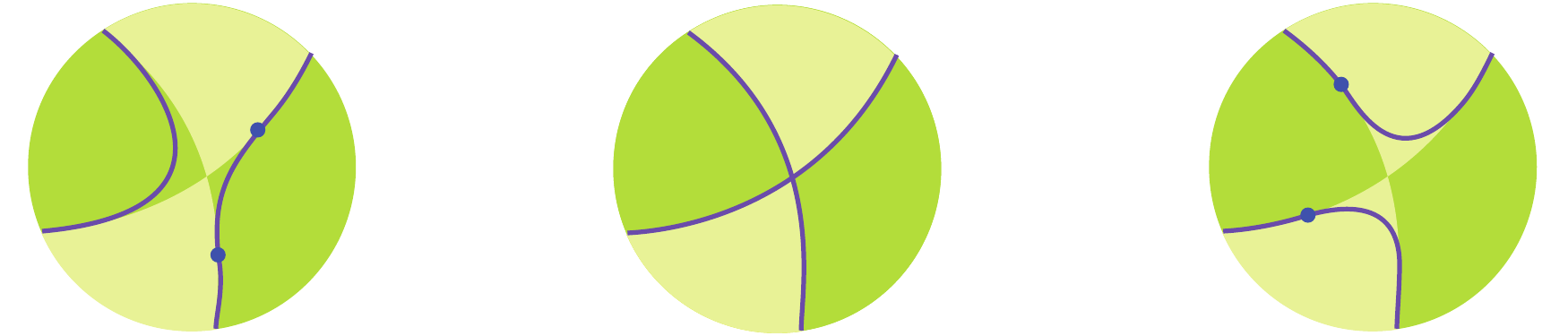}
   \caption{}\label{fig:x}
\end{figure}
Let us call the closure of each of these components a \emph{sector} of $\Gamma_1\cup\Gamma_2$. Each sector is bounded in the  interior of $B$ by a double spiral which is composed of one arm of $\Gamma_1$ and one arm of $\Gamma_2$, since by assumption $\Gamma_1$ and $\Gamma_2$ cross each other. Further, since by assumption $\Gamma_i$ intersect each other transversely, it follows that each sector is a proper side of the corresponding double spiral, as we defined in Section \ref{subsec:simple}.  Let us say that a sector is \emph{convex}, \emph{concave}, or \emph{semiconvex} according to whether the corresponding double spiral is convex, concave, or semiconvex respectively, again as defined in Section \ref{subsec:simple}. Now pick a pair of opposite sectors, i.e., a pair of sectors which do not share a common boundary spiral. Then there are only two possibilities: either both sectors are semiconvex, or one sector is convex while the other one is concave, as shown in  Figure \ref{fig:x}. In either case we may use Proposition \ref{prop:singularity} to remove the double point while adding at most two inflections. Finally note that one of these surgeries will leave $\Gamma$ connected (consider for instance the effect of these surgeries on the Figure $8$ curve). This completes the proof of Theorem \ref{thm:doublepoints}.

It only remains then to prove Proposition \ref{prop:separate}. First we show that after a perturbation we may assume that $\Gamma_i$  are transversal. To this end recall that we are assuming that $\Gamma_i$ cross each other, and  suppose that $\Gamma_i$ are tangent to each other at $o$. Then after a rotation, we may assume that $\Gamma_i$ are tangent to the $x$-axis. Let $f_i\colon [-a,a]\to\R$ be functions whose graphs coincide with neighborhoods of $o$ in $\Gamma_i$. Since $\Gamma_i$ cross each other, we may assume that $f_1\geq f_2$ on $[0,a]$, and $f_1\leq f_2$ on $[-a,0]$. Then the next lemma yields the desired perturbation:

\begin{lem}
Let $f\colon [-a,a]\to\R$ be a $\C^2$ function with $f(0)=f'(0)=0$, and $f''\neq 0$ on $[-a,0)\cup(0,a]$. Then there exists a $\C^2$ function $\tilde f\colon [-a,a]\to\R$ with $\tilde f(0)=0$, but $\tilde f'(0)\neq 0$ such that $\tilde f''\neq 0$ on $[-a,0)\cup(0,a]$, $\tilde f\equiv f$ near $\pm a$, $\tilde f\geq f$ on $[0,a]$, and $\tilde f\leq f$ on $[-a,0]$. Furthermore, if $f''(0)\neq 0$, then $\tilde f''(0)\neq 0$ as well.\qed
\end{lem}

 The proof of the above lemma is very similar to that of Lemma \ref{lem:rotatespiral} and so will be omitted. Now that $\Gamma_i$ are transversal, we may use Lemma \ref{lem:perturb} to finish the proof of Proposition \ref{prop:separate}. More specifically, we just need to perturb $\Gamma_i$ to make sure that they do not intersect at inflection points. To this end we may first perturb $\Gamma_1$ via Lemma \ref{lem:perturb} to make sure that the intersection of $\Gamma_i$  is not an inflection of $\Gamma_2$. Next we perturb $\Gamma_2$, again via Lemma \ref{lem:perturb}, near the new intersection point of $\Gamma_i$, which will ensure that the intersection of $\Gamma_i$ is not an inflection of $\Gamma_1$ either. This settles Proposition \ref{prop:separate}, and thus completes the proof of Theorem \ref{thm:doublepoints}.

\begin{note}
The basic patterns for the formation of inflections in Figure \ref{fig:x} appear to have been  observed by Klein \cite[p. 188]{klein2} in the context of real algebraic curves.
\end{note}

\section{Proof of the Main Result}\label{sec:proof}
Here we complete the proof of  Theorem \ref{thm:main2}. We may assume that the quantities $\mathcal{D}$, $\mathcal{S}$, and $\mathcal{I}$ are all finite. Then  $\Gamma$ is a $\C^2$  curve with a finite number of singularities. Further, the finiteness of $\mathcal{I}$ implies that $\Gamma$ does not contain any geodesic arcs, and so cannot be contained entirely in a plane passing through the  origin. In particular, $\conv(\Gamma)$ has interior points. Recall also that we may assume that all points of $\Gamma$ which are not simple are double points, for otherwise $\mathcal{D}^+\geq 3$ and  there would be nothing left to prove. 
Next we show that we may assume that all the antipodal pairs of points of $\Gamma$ are simple and regular. To this end we need one more  fact concerning spherical curves which is a finer version of a lemma of Fenchel \cite{fenchel2,totaro}, since it assumes less regularity. 

\begin{lem}\label{lem:antipodal}
Let $\Gamma\subset\S^2$ be a simple arc which is $\C^2$-regular in its interior. If $\Gamma$ has no inflections in its interior, then it lies in an open hemisphere.
\end{lem}
\begin{proof}
Let $p$ be the initial boundary point of $\Gamma$, with respect to some fixed orientation, and $q$ be the last point in $\Gamma$ such that the  subarc $pq$   lies in a hemisphere $H$. Note that $pq$ lies in an open hemisphere, only if $q$ is the final boundary point of $\Gamma$, in which case we are done. Suppose then, towards a contradiction, that $pq$ does not lie in any open hemisphere. Then $pq$ must meet the boundary $C$ of $H$ in at least two points. If there are exactly two such points, then they must be antipodal, say $\pm r$. Consider the subarc $\ol\Gamma\subset\Gamma$ bounded by $\pm r$. Then $\ol\Gamma$ is $\C^1$-regular by Lemma \ref{lem:spiralball}, hence Lemma \ref{lem:hemisphere} applies to $\ol\Gamma$. Indeed, after rotating $C$ about $\pm r$, we may assume that $\ol\Gamma$ is tangent to $C$ at some point. Then Lemma \ref{lem:hemisphere} implies that  $\ol\Gamma$  has an inflection in its interior, which is impossible. So  $pq$ must meet $C$ in at least $3$ points, say $r$, $s$, $t$, and suppose that $s$ lies in the oriented arc $rt$.  Then $rt$ is again $\C^1$-regular by Lemma \ref{lem:spiralball}. Furthermore, $rt$ must be tangent to $C$ at $s$, and it follows once more from Lemma \ref{lem:hemisphere} that at least one of the  subarcs $rs$ or $st$  has an inflection in its interior, which is again impossible. 
\end{proof}

Let $\Sigma^+(\Gamma)$ be as in \eqref{eq:sig+}, and set
\begin{equation*}
\Sigma(\Gamma):=2(\mathcal{D}+\mathcal{S})+\mathcal{I}.
\end{equation*}
Lemma \ref{lem:antipodal} quickly yields: 

\begin{cor}\label{cor:antipodal}
Let $\Gamma\subset\S^2$ be a $\C^2$ closed curve, and $\pm p$ be a pair of antipodal points of $\Gamma$. If $\pm p$ are not both regular simple  points of $\Gamma$, then  $\Sigma(\Gamma)\geq 6$ and $\Sigma^+(\Gamma)\geq 4$.
\end{cor}
\begin{proof}
Consider the subarcs of $\Gamma$ with end points $\pm p$. By Lemma \ref{lem:antipodal} each of these arcs must have either an inflection, a singularity, or a double point in its interior. So it follows that $\Sigma^+(\Gamma)\geq 2$. Now if either $p$ or $-p$ is a singular or a double point of $\Gamma$, then we have $\Sigma^+(\Gamma)\geq 4$ as desired. Further, since $\pm p\in\Gamma$, we also have $\Sigma(\Gamma)\geq \Sigma^+(\Gamma)+2\geq 6$, which completes the proof.
\end{proof}

Now we are ready to finish the proof of Theorem \ref{thm:main2}. First note that
if $\Gamma$ is symmetric, then by the last observation we may assume that it has no singularities or double points,  because every point of $\Gamma$ is antipodal to another one. Then Proposition \ref{prop:DS} completes the proof of the inequality \eqref{eq:2.75}. It remains then to prove Inequalities \eqref{eq:2} and \eqref{eq:2.5}.   There are  two cases to consider: either  $o$ is an interior point of $\conv(\Gamma)$, or $o$ lies on the boundary of $\conv(\Gamma)$.

Suppose first that $o$ lies in the interior of  $\conv(\Gamma)$. Then $o$ remains in $\conv(\Gamma)$ after small perturbations of $\Gamma$. Thus, using Theorems \ref{thm:singularity} and \ref{thm:doublepoints}, we may remove the singularities and double points of $\Gamma$ while keeping $o$ in $\conv(\Gamma)$, and without increasing $\Sigma^+(\Gamma)$. Furthermore, $\Sigma(\Gamma)$ will not increase either, since by Corollary \ref{cor:antipodal} we may assume that the antipodal points of $\Gamma$ are all simple and regular. Thus we do not need to perturb $\Gamma$ near those points (which might risk  increasing the number of antipodal pairs of points of $\Gamma$). In short, we may assume that $\mathcal{D}^+=\mathcal{S}=0$, in which case Proposition \ref{prop:DS} completes the proof.

It remains now to consider the case where $o$ lies on the boundary of $\conv(\Gamma)$, or equivalently $\Gamma$  lies in a hemisphere  $H$ of $\S^2$. Since $\Gamma\subset H$ and $o\in\conv(\Gamma)$, 
$\Gamma$ must intersect the great circle $C$ which bounds $H$. There are only two cases to consider: either $\Gamma\cap C$ contains a pair of antipodal points, or not.

Suppose first that  $\Gamma\cap C$  contains a pair of antipodal points $\pm p$.   By Corollary \ref{cor:antipodal}, $\pm p$ must be simple regular points of $\Gamma$. Now using Theorems \ref{thm:singularity} and \ref{thm:doublepoints}, we may remove the singularities and double points of $\Gamma$ without perturbing $\pm p$, and thus keeping $o$ in $\conv(\Gamma)$. Recall that these operations do not increase $\Sigma^+(\Gamma)$. Furthermore,  $\Sigma(\Gamma)$ does not increase either, because all the singularities and double points of $\Gamma$ are away from $-\Gamma$, by Corollary \ref{cor:antipodal}, and thus perturbations near these points will not create any new pairs of antipodal points. Now that we have eliminated all the singularities and double points, Proposition  \ref{prop:DS}  completes the proof in this case.

Finally suppose that $\Gamma\cap C$ does not contain any pairs of antipodal points. Then $\Gamma\cap C$ must contain  three points $p_1$, $p_2$, $p_3$, which contain  $o$ in the relative interior of their convex hull (this follows from Carath\'{e}odory's theorem, as we  argued in the proof of Corollary \ref{cor:hemisphere}).  Note that at least one of these three points, say $p_1$, must be a regular simple point of $\Gamma$, for otherwise $\Sigma^+(\Gamma)\geq 6$ and we are done. Furthermore,  $-p_1\not\in\Gamma$ by assumption. Now we may use Lemma \ref{lem:perturb} to perturb $\Gamma$ near $p_1$, and  without increasing $\Sigma(\Gamma)$ or $\Sigma^+(\Gamma)$, so that $o$ lies in the interior of $\conv(\Gamma)$.  Indeed, since  we are assuming that $\Gamma$ is not a great circle, there is a  point $p_0\in \Gamma$ which does not lie in the plane of $C$. So $o$ lies in the tetrahedron $T$ formed by $\{p_0,\dots,p_3\}$. Suppose that $C$ lies in the $xy$-plane and $T$ lies ``above" this plane. Then by construction $o$ lies in the interior of the base of $T$. So it is clear that if we move  $p_1$ by a small distance below the $xy$-plane, then $o$ will lie in the interior of $T$, and therefore in the interior of $\conv(\Gamma)$, as claimed. Since we have already treated  this case, the proof is now complete.

\section{Sharp Examples}\label{sec:examples}
Here we describe some examples which illustrate the sharpness of the inequalities in Theorem \ref{thm:main2}, and consequently establish the sharpness of all the inequalities cited in the introduction.

 We begin with inequality \eqref{eq:2}. Note that there are exactly $10$  triples $(\mathcal{D}, \mathcal{S}, \mathcal{I})$ of nonnegative integers such that $2(\mathcal{D}+\mathcal{S})+\mathcal{I}=6$. Each of these cases may be realized by a spherical curve satisfying the hypothesis of Theorem \ref{thm:main2}. To generate these examples we begin with the curve depicted in Figure \ref{fig:examples} (a). 
 \begin{figure}[h] 
   \centering
   \begin{overpic}[height=1.5in]{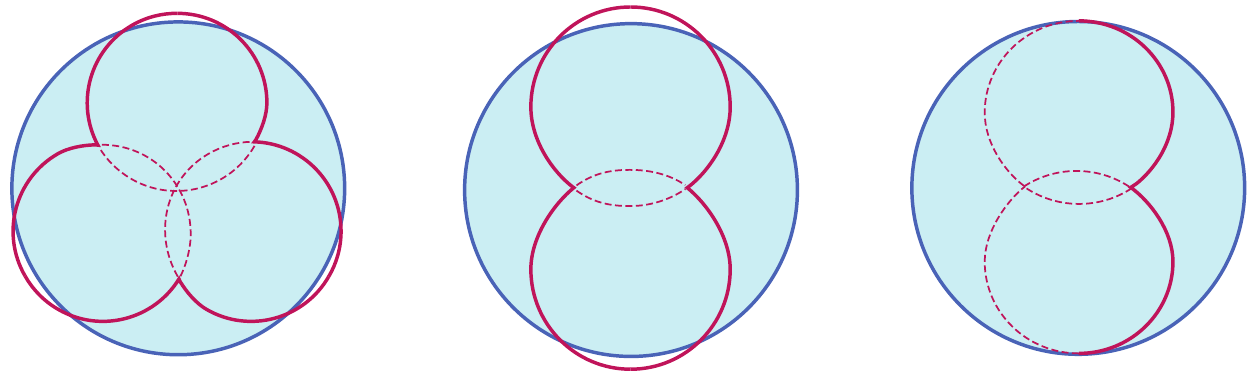} 
   \put(12,-3){$(a)$}
   \put(49,-3){$(b)$}
   \put(85,-3){$(c)$}
   \end{overpic}
   \caption{}
   \label{fig:examples}
\end{figure}
This picture illustrates the  stereographic projection of a spherical curve from the ``North pole" $(0,0,1)$ of $\S^2$ into the tangent plane at the ``South pole" $(0,0,-1)$, and the shaded disk corresponds to the ``Southern hemisphere". Recall that stereographic projections preserve circles. Thus the curve in Figure \ref{fig:examples} (a) which is composed of $3$ circular arcs has no inflections ($\mathcal{I}=0$). Further, it is not hard to see that this curve is disjoint from its antipodal reflection, and so $\mathcal{D}=0$, since this curve is also simple. Furthermore, this curve  contains the origin $o$ in its convex hull, because it intersects the equator in $6$ points which form a polygon passing through $o$. We conclude then that the curve in Figure \ref{fig:examples} (a) is an example of  type $(0,3,0)$. To generate examples of other types, note that we may replace each singularity of this curve with a small loop, see Figure \ref{fig:cusp}. 
 \begin{figure}[h] 
   \centering
    \begin{overpic}[height=0.9in]{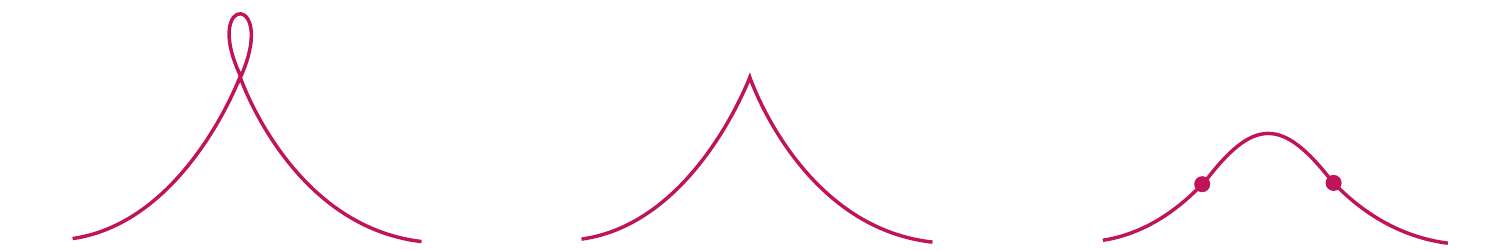} 
    \end{overpic}
   \caption{}
   \label{fig:cusp}
\end{figure}
Alternatively, as Figure \ref{fig:cusp} demonstrates, each singularity may be smoothed away at the cost of precisely two inflections as we proved in Proposition \ref{prop:singularity}. Thus we obtain examples which exhaust all possible $10$ cases where equality may hold in \eqref{eq:2}.

 Similarly, one may establish the sharpness of inequality \eqref{eq:2.5}. Here there are only $6$ cases to consider, since that is precisely the number of possible triplets $(\mathcal{D}^+,\mathcal{S}, \mathcal{I})$ of nonnegative integers such that $2(\mathcal{D}^++\mathcal{S})+\mathcal{I}=4$. Again each case will be possible. The case $(0,2,0)$ is illustrated in Figure \ref{fig:examples} (b), and the other cases may be generated by replacing the singularities by small loops, or a  pair of inflections as described earlier.

 Finally, we consider inequality \eqref{eq:2.75}. Here each of the quantities in the triple $(\mathcal{D}^+, \mathcal{S}, \mathcal{I})$ must be even due to symmetry. Thus, for the equality to hold in \eqref{eq:2.75}, there are only $3$ possible cases: $(0,0,6)$, $(0,2,2)$, and $(2, 0, 2)$. Again, all these cases may be realized. Figure \ref{fig:examples} (c) corresponds to the case $(0,2,2)$. This picture depicts one-half of the curve in a hemisphere. Here the depicted arcs are not exactly circular, but have been perturbed slightly so that they have contact of order $2$ with the boundary of the hemisphere. Now, gluing this curve to its antipodal reflection yields a $\C^2$-regular closed curve of type $(0,2,2)$. The other cases $(2, 0, 2)$ and $(0,0,6)$ may again be treated by replacing the singularity in Figure \ref{fig:examples} (c) with a small loop or a pair of inflections.
 
 \begin{note}\label{note:fr}
 Each of the spherical curves (or tantrices) we constructed here may be integrated to obtain a closed space curve, since these curves all contain the origin in the interior of their convex hull \cite{ghomi:knots}. Thus we will obtain sharp examples for inequalities \eqref{eq:1} and \eqref{eq:1.5}. It would be interesting to know if such examples can be constructed in every knot class, or  whether the lower bound for these inequalities may be improved according to the topological complexity of the curve. Alternatively, one could also ask whether we may improve or balance these inequalities by adding some geometric terms to their right hand sides.  See the paper of Wiener \cite{weiner:fabricius} who adapted the Fabricius-Bjerre formula \cite{fr,halpern} to spherical and space curves. Those results, however, assume that the curves are generic and do not give estimates for the total number of pairs of parallel tangent lines, but only the difference between the number of concordant and discordant pairs.
 \end{note}
 
 \begin{note}
If  a closed  space curve which is generic with respect to the $\C^1$-topology has a pair of discordant parallel tangent lines, then it must have at least two such pairs. This is due to the basic observation that
if a simple closed curve $\Gamma\subset\S^2$  crosses its antipodal reflection $\Gamma'$, i.e., $\Gamma$ does not lie on one side of $\Gamma'$, then $\Gamma$ must have at least two pairs of antipodal points.    To see this,
suppose that $p$ and $p'$ are the  only  antipodal pairs of points of $\Gamma$. Let $\Omega$ be a region bounded by $\Gamma$ in $\S^2$ with area at most $2\pi$. Next, let $A$ and $B$ be the two subarcs of $\Gamma$ determined by $p$ and $p'$. Then the reflection of one of these arcs, say  $A$, which we denote by $A'$, must lie in $\Omega$, since by assumption $\Gamma'$ must enter $\Omega$. Now note that the closed curve $A\cup A'$ is simple and symmetric. Thus it bisects $\S^2$. In particular, the region $\Omega'$ in $\Omega$ which is bounded by $A\cup A'$ must have area $2\pi$. So it follows that $\Omega=\Omega'$, and in particular $\Gamma=\Gamma'$, which  contradicts the assumption that $\Gamma$ crosses $\Gamma'$.
\end{note}

\begin{note}\label{note:normals}
An analogue of Theorem \ref{thm:main} for the number of parallel normal lines of a closed space curve may be obtained as follows.
 If $\gamma\colon\S^1\to\R^3$ is a $\C^2$ unit speed curve with nonvanishing curvature $\kappa$, then its \emph{principal normal} vector field $N\colon\S^1\to\S^2$, given by $N:=T'/\kappa$, is a well-defined spherical curve. Since $T'=\kappa N$,  $\gamma$ has  parallel normal lines at a pair of points, if, and only if, $T$ has parallel tangent lines at the corresponding points. Hence Theorem \ref{thm:main2} may be used to obtain information with regard to the number of parallel normal lines of a space curve.
Recall that $\tau/\kappa$ is the geodesic curvature of  $T$. The  critical points of this quantity, i.e., the zeros of $(\tau/\kappa)'$, have been called \emph{Darboux vertices} \cite{heil}. Now suppose that $\gamma$ is $\C^4$, let
 $\mathcal{P}_N$ be the number of (unordered) pairs of   points $t\neq s$ in $\S^1$ where normal lines of $\gamma$ are parallel, i.e., $N(t)=\pm N(s)$, and   $\mathcal{V}_d$ be the number of Darboux vertices of $\gamma$. We claim that 
\begin{equation}\label{eq:darboux}
2\mathcal{P}_N+\mathcal{V}_d\geq 6.
\end{equation}
To establish this inequality, first note that $N$ is the tantrix of $T$, which is a closed space curve. Thus $N$ contains the origin in its convex hull, and so 
Theorem \ref{thm:main2} may be applied to $N$. Further, assuming that $\gamma$ is parametrized by arclength, we have $N'=-\kappa T +\tau B$, where $B:=T\times N$ is the binormal vector. Thus $\|N'\|=\kappa^2+\tau^2\neq 0$, which means that $N$ has no singularities. So, by Theorem \ref{thm:main2}, $2\mathcal{P}_N+\mathcal{I}\geq 6$, where $\mathcal{I}$ is the number of inflections of $N$. It remains only to note then that
 the geodesic curvature of $N$ vanishes if and only if $\langle N'',N\times N'\rangle$ vanishes, and a straight forward computation shows that this is the case precisely when $(\tau/\kappa)'$ vanishes.  Finally, we should mention that we do not know whether \eqref{eq:darboux} is sharp. A major difficulty in constructing examples here is that, in contrast to the tantrix which has a simple characterization, no complete characterization for the normal spherical image of a closed space curve is known \cite{fenchel}.
\end{note}

\section*{Acknowledgements}
The author thanks  Xiang Ma, Gaiane Panina, John Pardon, Bruce Solomon, Serge Tabachnikov, and Masaaki Umehara  for useful communications.

\bibliographystyle{abbrv}
\bibliography{references}

\end{document}